\documentclass[12pt]{amsart}
\usepackage{amssymb,amsthm,amsmath,amsfonts,latexsym,tikz,hyperref,subfigure}

\usepackage[hmargin=1in,vmargin=1in]{geometry}
\usepackage{verbatim}
\usepackage{graphicx,tikz}
\usepackage{hyperref}
\usepackage[utf8]{inputenc}
\usepackage{cancel,ulem}
\usepackage{caption}
\usepackage{multirow}
\usepackage{xcolor}

\usetikzlibrary{shapes.geometric,positioning}
\tikzstyle{vertex}=[circle, draw, inner sep=0pt, minimum size=2pt]
\tikzset{every loop/.style={}} 

\newtheorem{theorem}{Theorem}[section]
\newtheorem{proposition}[theorem]{Proposition}
\newtheorem{lemma}[theorem]{Lemma}

\newtheorem{corollary}[theorem]{Corollary}

\theoremstyle{definition}
\newtheorem{definition}[theorem]{Definition} 
\newtheorem{example}[theorem]{Example}
\newtheorem{remark}[theorem]{Remark}

\newtheorem{question}{Question} 
\newtheorem{oquestion}{Open Problem}
 
\newcommand{\tcr}[1]{\textcolor{red}{#1}}
\newcommand{\tcb}[1]{\textcolor{blue}{#1}}
\newcommand{\tco}[1]{\textcolor{orange}{#1}}
\newcommand{\bb}{\mathbf{b}}
\newcommand{\pp}{\mathbf{p}}
\newcommand{\cc}{\mathbf{c}}

\newcommand{\Bn}{B_\ell}
\newcommand{\Bl}{B_{\ell,1}}

\newcommand{\Klo}{K_\ell^\circ}
\newcommand{\Cn}{C_\ell}
\renewcommand{\k}{j}
\newcommand{\Wmnk}{\mathbf{W}_{n,m,k}}
\newcommand{\Tmnk}{\mathbf{T}_{m+n,m-k,k}}
\newcommand{\M}{\textup{max}}
\newcommand{\Tlbk}{\mathbf{T}_{\ell,b,k}}

\newcommand{\rr}{r}
\newcommand{\rlv}{\vec{r}}

\renewcommand{\aa}{\mathbf{a}}
\newcommand{\touch}{\mathrm{t}}
\newcommand{\tv}{\vec{t}}

\usetikzlibrary{decorations.markings}
\newcommand{\NEpath}[4]{
    \fill[white!25]  (#1) rectangle +(#2,#3);
    \fill[fill=white]
    (#1)
    \foreach \dir in {#4}{
        \ifnum\dir=1
        -- ++(1,0)
        \else
        -- ++(0,1)
        \fi
    } |- (#1);
    \draw[help lines] (#1) grid +(#2,#3);
    \draw[dashed] (#1) -- +(#3,#3);
    \coordinate (prev) at (#1);
    \foreach \dir in {#4}{
        \ifnum\dir=1
        \coordinate (dep) at (1,0);
        \else
        \coordinate (dep) at (0,1);
        \fi
        \draw[line width=2pt,-stealth] (prev) -- ++(dep) coordinate (prev);
    };
}
 
\title{A combinatorial model for lane merging}  
 
\author[Bardenova]{Viktoriya Bardenova}
\address[V.~Bardenova]{Department of Mathematics, Florida Gulf Coast University, Fort Myers,FL}
\email{\textcolor{blue}{\href{mailto:vlbardenova6187@eagle.fgcu.edu}{vlbardenova6187@eagle.fgcu.edu}}}
\author[Insko]{Erik Insko}
\address[E.~Insko]{Department of Mathematics, Florida Gulf Coast University, Fort Myers, FL}
\email{\textcolor{blue}{\href{mailto:einsko@fgcu.edu}{einsko@fgcu.edu}}}
\author[Johnson]{Katie Johnson}
\address[K.~Johnson]{Department of Mathematics, Florida Gulf Coast University, Fort Myers, FL}
\email{\textcolor{blue}{\href{mailto:kjohnson@fgcu.edu}{kjohnson@fgcu.edu}}}
\author[Sullivan]{Shaun Sullivan}
\address[S.~Sullivan]{Department of Mathematics, Florida Gulf Coast University, Fort Myers, FL}
\email{\textcolor{blue}{\href{mailto:ssullivan@fgcu.edu}{ssullivan@fgcu.edu}}}
\date{\today}

\begin{document}

\maketitle

\begin{abstract}
    A two lane road approaches a stoplight. The left lane merges into the right just past the intersection. Vehicles approach the intersection one at a time, with some drivers always choosing the right lane, while others always choose the shorter lane, giving preference to the right lane to break ties. An arrival sequence of vehicles can be represented as a binary string, where the zeros represent drivers always choosing the right lane, and the ones represent drivers choosing the shorter lane. From each arrival sequence we construct a merging path, which is a lattice path determined by the lane chosen by each car.  We give closed formulas for the number of merging paths reaching the point $(n,m)$ with exactly $k$ zeros in the arrival sequence, and the expected length of the right lane for all arrival sequences with exactly $k$ zeros. Proofs involve an adaptation of Andre's Reflection Principle. Other interesting connections also emerge, including to: Ballot numbers, the expected maximum number of heads or tails appearing in a sequence of $n$ coin flips, the largest domino snake that can be made using pieces up to $[n:n]$, and the longest trail on the complete graph $K_n$ with loops.
\end{abstract}

\section{Introduction}

Imagine you are driving on a road with two lanes where there is a stoplight and soon after, the left lane will have to merge into the right. Some drivers will move to the right lane before the traffic light, regardless of its length. Others will choose the shortest lane, giving preference to the right lane when the lengths are equal.

We model this situation using a binary string called an {\it arrival sequence}, assuming cars approach the stoplight one at a time with plenty of time to choose their preferred lane. Cars that do not want to merge and that will always choose the right lane are denoted with $0$ and colored red in diagrams. Cars that prefer the shortest lane (with ties going to the right lane) are denoted by $1$ and colored green.

In Figure \ref{fig:cars}, the arrival sequence is $\bb=0011\tcb{1}001.$ The first car will always choose the right lane, no matter what. In this case, the second car is red and will also choose the right lane. The next three cars are green; two will choose the left lane and the third will choose the right as the lanes will be equal in length at that point. (Green cars that end up in the right lane anyway will appear as blue digits in arrival sequences throughout the paper, for extra clarity.) Cars 6 and 7 are red and will choose the right lane. Finally, car 8 is green and will choose the left lane.

\begin{figure}[]
\centering
\begin{subfigure}{ }
    \centering
    \begin{tikzpicture}[scale=0.9]
 \begin{scope}[scale=0.5]
     \shade[top color=red, bottom color=white, shading angle={45}]
    [draw=black,fill=red!20,rounded corners=1.2ex,very thick] (1.5,.5) rectangle (6.5,1.8);
    \draw[very thick, rounded corners=0.5ex,fill=black!20!blue!20!white,thick]  (2.5,1.8) -- ++(1,0.7) -- ++(1.6,0) -- ++(0.6,-0.7) -- (2.5,1.8);
    \draw[thick]  (4.2,1.8) -- (4.2,2.5);
    \draw[draw=black,fill=gray!50,thick] (2.75,.5) circle (.5);
    \draw[draw=black,fill=gray!50,thick] (5.5,.5) circle (.5);
    \draw[draw=black,fill=gray!80,semithick] (2.75,.5) circle (.4);
    \draw[draw=black,fill=gray!80,semithick] (5.5,.5) circle (.4);
    \draw[-,semithick] (0,3.5) -- (7,3.5);
    \draw[dashed,thick] (0,-0.5) -- (7,-0.5);
    \draw (4,1.2) node {1};
  \end{scope}
   \begin{scope}[xshift=100,scale=0.5] 
    \shade[top color=red, bottom color=white, shading angle={45}]
    [draw=black,fill=red!20,rounded corners=1.2ex,very thick] (1.5,.5) rectangle (6.5,1.8);
    \draw[very thick, rounded corners=0.5ex,fill=black!20!blue!20!white,thick]  (2.5,1.8) -- ++(1,0.7) -- ++(1.6,0) -- ++(0.6,-0.7) -- (2.5,1.8);
    \draw[thick]  (4.2,1.8) -- (4.2,2.5);
    \draw[draw=black,fill=gray!50,thick] (2.75,.5) circle (.5);
    \draw[draw=black,fill=gray!50,thick] (5.5,.5) circle (.5);
    \draw[draw=black,fill=gray!80,semithick] (2.75,.5) circle (.4);
    \draw[draw=black,fill=gray!80,semithick] (5.5,.5) circle (.4);
    \draw[-,semithick] (0,3.5) -- (7,3.5);
    \draw[dashed,thick] (0,-0.5) -- (7,-0.5);
    \draw (4,1.2) node {2};
  \end{scope}
  \begin{scope}[xshift=200,scale=0.5] 
    \shade[top color=green, bottom color=white, shading angle={45}]
    [draw=black,fill=red!20,rounded corners=1.2ex,very thick] (1.5,.5) rectangle (6.5,1.8);
    \draw[very thick, rounded corners=0.5ex,fill=black!20!blue!20!white,thick]  (2.5,1.8) -- ++(1,0.7) -- ++(1.6,0) -- ++(0.6,-0.7) -- (2.5,1.8);
    \draw[thick]  (4.2,1.8) -- (4.2,2.5);
    \draw[draw=black,fill=gray!50,thick] (2.75,.5) circle (.5);
    \draw[draw=black,fill=gray!50,thick] (5.5,.5) circle (.5);
    \draw[draw=black,fill=gray!80,semithick] (2.75,.5) circle (.4);
    \draw[draw=black,fill=gray!80,semithick] (5.5,.5) circle (.4);
    \draw[-,semithick] (0,3.5) -- (7,3.5);
    \draw[dashed,thick] (0,-0.5) -- (7,-0.5);
    \draw (4,1.2) node {5};
  \end{scope}
  \begin{scope}[xshift=300,scale=0.5] 
    \shade[top color=red, bottom color=white, shading angle={45}]
    [draw=black,fill=red!20,rounded corners=1.2ex,very thick] (1.5,.5) rectangle (6.5,1.8);
    \draw[very thick, rounded corners=0.5ex,fill=black!20!blue!20!white,thick]  (2.5,1.8) -- ++(1,0.7) -- ++(1.6,0) -- ++(0.6,-0.7) -- (2.5,1.8);
    \draw[thick]  (4.2,1.8) -- (4.2,2.5);
    \draw[draw=black,fill=gray!50,thick] (2.75,.5) circle (.5);
    \draw[draw=black,fill=gray!50,thick] (5.5,.5) circle (.5);
    \draw[draw=black,fill=gray!80,semithick] (2.75,.5) circle (.4);
    \draw[draw=black,fill=gray!80,semithick] (5.5,.5) circle (.4);
    \draw[-,semithick] (0,3.5) -- (7,3.5);
    \draw[dashed,thick] (0,-0.5) -- (7,-0.5);
    \draw (4,1.2) node {6};
  \end{scope}
  \begin{scope}[xshift=400,scale=0.5] 
    \shade[top color=red, bottom color=white, shading angle={45}]
    [draw=black,fill=red!20,rounded corners=1.2ex,very thick] (1.5,.5) rectangle (6.5,1.8);
    \draw[very thick, rounded corners=0.5ex,fill=black!20!blue!20!white,thick]  (2.5,1.8) -- ++(1,0.7) -- ++(1.6,0) -- ++(0.6,-0.7) -- (2.5,1.8);
    \draw[thick]  (4.2,1.8) -- (4.2,2.5);
    \draw[draw=black,fill=gray!50,thick] (2.75,.5) circle (.5);
    \draw[draw=black,fill=gray!50,thick] (5.5,.5) circle (.5);
    \draw[draw=black,fill=gray!80,semithick] (2.75,.5) circle (.4);
    \draw[draw=black,fill=gray!80,semithick] (5.5,.5) circle (.4);
    \draw[-,semithick] (0,3.5) -- (7,3.5);
    \draw[dashed,thick] (0,-0.5) -- (7,-0.5);
    \draw (4,1.2) node {7};
  \end{scope}
\begin{scope}[yshift=-50,scale=0.5]
     \shade[top color=green, bottom color=white, shading angle={45}]
    [draw=black,fill=red!20,rounded corners=1.2ex,very thick] (1.5,.5) rectangle (6.5,1.8);
    \draw[very thick, rounded corners=0.5ex,fill=black!20!blue!20!white,thick]  (2.5,1.8) -- ++(1,0.7) -- ++(1.6,0) -- ++(0.6,-0.7) -- (2.5,1.8);
    \draw[thick]  (4.2,1.8) -- (4.2,2.5);
    \draw[draw=black,fill=gray!50,thick] (2.75,.5) circle (.5);
    \draw[draw=black,fill=gray!50,thick] (5.5,.5) circle (.5);
    \draw[draw=black,fill=gray!80,semithick] (2.75,.5) circle (.4);
    \draw[draw=black,fill=gray!80,semithick] (5.5,.5) circle (.4);
    \draw[-,semithick] (0,-0.5) -- (7,-0.5);
    \draw (4,1.2) node {3};
  \end{scope}
   \begin{scope}[xshift=100,yshift=-50,scale=0.5] 
    \shade[top color=green, bottom color=white, shading angle={45}]
    [draw=black,fill=red!20,rounded corners=1.2ex,very thick] (1.5,.5) rectangle (6.5,1.8);
    \draw[very thick, rounded corners=0.5ex,fill=black!20!blue!20!white,thick]  (2.5,1.8) -- ++(1,0.7) -- ++(1.6,0) -- ++(0.6,-0.7) -- (2.5,1.8);
    \draw[thick]  (4.2,1.8) -- (4.2,2.5);
    \draw[draw=black,fill=gray!50,thick] (2.75,.5) circle (.5);
    \draw[draw=black,fill=gray!50,thick] (5.5,.5) circle (.5);
    \draw[draw=black,fill=gray!80,semithick] (2.75,.5) circle (.4);
    \draw[draw=black,fill=gray!80,semithick] (5.5,.5) circle (.4);
    \draw[-,semithick] (0,-0.5) -- (7,-0.5);
    \draw (4,1.2) node {4};
  \end{scope}
  \begin{scope}[xshift=200,yshift=-50,scale=0.5] 
    \shade[top color=green, bottom color=white, shading angle={45}]
    [draw=black,fill=red!20,rounded corners=1.2ex,very thick] (1.5,.5) rectangle (6.5,1.8);
    \draw[very thick, rounded corners=0.5ex,fill=black!20!blue!20!white,thick]  (2.5,1.8) -- ++(1,0.7) -- ++(1.6,0) -- ++(0.6,-0.7) -- (2.5,1.8);
    \draw[thick]  (4.2,1.8) -- (4.2,2.5);
    \draw[draw=black,fill=gray!50,thick] (2.75,.5) circle (.5);
    \draw[draw=black,fill=gray!50,thick] (5.5,.5) circle (.5);
    \draw[draw=black,fill=gray!80,semithick] (2.75,.5) circle (.4);
    \draw[draw=black,fill=gray!80,semithick] (5.5,.5) circle (.4);
    \draw[-,semithick] (0,-0.5) -- (7,-0.5);
    \draw (4,1.2) node {8};
  \end{scope}
  \begin{scope}[xshift=300,yshift=-50,scale=0.5]
   \draw[-,semithick] (0,-0.5) -- (14,-0.5);
  \end{scope}
  \draw [very thick] [<-](0,-2.5)--(2,-2.5) node[right]{\text{Direction of traffic}};
\end{tikzpicture}
    \end{subfigure} \hspace{3cm}
    \begin{subfigure}{ }
 \centering 
 \begin{tikzpicture}[scale=1.1]
    \NEpath{0,0}{5}{5}{0,0,1,1,0,0,0,1};
    \draw[blue, line width=2pt,-stealth] (2,2) -- ++(0,1);
    \end{tikzpicture}
    \end{subfigure}
    \caption{Eight cars waiting to merge and the corresponding merging path for $\bb=0011\tcb{1}001.$}
    \label{fig:cars}
\end{figure}

To each arrival sequence, we can assign a (decorated) lattice path, which we call a {\it merging path}.  For instance, the merging path for the arrival sequence $\bb=0011\tcb{1}001$ is also shown in Figure~\ref{fig:cars}. When a green car ends up staying in the right lane, we say that the merging path {\it bounces} off the diagonal, and we decorate the corresponding upward step by highlighting it in blue. Merging paths without any bounces are the famous \textit{ballot paths} (i.e.~lattice paths that do not cross below the diagonal). Hence, merging paths generalize the ballot paths  which are used to enumerate the number of ways the ballots in a two-candidate election can be counted so that the winning candidate remains in the lead at all times \cite{AR08,A87,NS10,R08}.  

An excellent and thorough history of lattice path enumeration is available in Humphreys' survey paper from 2010 \cite{H10}, which also includes a discussion of the reflection principle that is used in this paper. Humphreys provides context and further reading for applications as wide-ranging as games \cite{W01} to electrostatics \cite{M92}, number theory \cite{M93,S18} to statistics \cite{K03, N79}. We remark that there is a rich history of discovering bijections between lattice paths and other mathematical objects \cite{K65,M79,S12}.

Before proceeding, let's set our notation conventions regarding arrival sequences and merging paths.  We denote the total number of cars in the arrival sequence by $\ell$, and let $B_\ell$ denote the set of all arrival sequences of length $\ell$. We denote the final number of cars in the right lane of an arrival sequence $\bb$ by $\rr(\bb)$, and then the length of the left lane must be $\ell-\rr$. The number of zeros in an arrival sequence will be denoted $k$, and the number of ones must be $\ell-k.$ In the previous example $\ell=8$, $r=5$, and $k=4.$

The original motivating questions we answer in this paper are:

\begin{question} \label{question:1}
What is the expected length $\mathbb{E}[\ell]$ of the right lane when we consider all possible arrival sequences of length $\ell$?
\end{question}

\begin{question} \label{question:2}
What is the expected length $\mathbb{E}[\ell,k]$ of the right lane when we consider all possible arrival sequences of length $\ell$ containing exactly $k$ zeros?
\end{question}

\begin{example}
When $\ell=2$ the collection of arrival sequences is $B_2=\{00, 10,01, 11\}$, and we calculate the sum of the right lane lengths as $R(B_2) = \sum_{\bb\in B_2} \rr(\bb) = 2+2+1+1= 6.$ So $\mathbb{E}[2]=R(B_2)/2^2=1.5$.
When $\ell=3$, $R(B_3)=\sum_{\bb\in B_3} \rr(\bb) = 18$, and the expected length of the right lane in a randomly selected arrival sequence is $\mathbb{E}[3]=R(B_3)/2^3=2.25$.
\end{example}

\begin{example}
The collection of arrival sequences of length $\ell=4$ with exactly $k=2$ zeros is $B_{4,2}=\{0011, 0101, 0110, 1001, 1010, 1100\}.$ In this case, $$ R(B_{4,2}) = \sum_{\bb\in B_{4,2}} \rr(\bb) = 2+2+3+3+3+3=16.$$ The expected length of the right lane of a randomly selected arrival sequence in $B_{4,2}$ is $\mathbb{E}[4,2]=R(B_{4,2})/\binom{4}{2}=\frac{8}{3}$.
\end{example}

Section~\ref{sec:mergingpaths} is dedicated to answering Question~\ref{question:1}.  To do this, we count the number of merging paths that begin at $(0,0)$ and end at $(n,m)$. We call the number of such merging paths $M_{n}(m)$.  Our first main result, Theorem~\ref{theorem:closed},  describes closed formulas for the numbers $M_n(m)$ in terms of binomial coefficients.  We then use Theorem~\ref{theorem:closed} to prove Theorem~\ref{thm:E_l} which answers Question~\ref{question:1}. We show that as $\ell$ tends to infinity, $\mathbb{E}[\ell]/\ell$  tends to $\frac{1}{2}$ in Corollary~\ref{corollary:limit}. 

\begin{figure}[h!]
\begin{tikzpicture}
    \NEpath{0,0}{5}{5}{0,0,0,1,1,1,0,0,1,1};
    \draw[blue, line width=2pt,-stealth] (0,0) -- ++(0,1);
    \end{tikzpicture} \hspace{12pt}
\begin{tikzpicture}
    \NEpath{0,0}{5}{5}{0,1,0,1,0,1,0,0,1,1};
    \draw[blue, line width=2pt,-stealth] (1,1) -- ++(0,1);
    \draw[blue, line width=2pt,-stealth] (3,3) -- ++(0,1);
    \end{tikzpicture}
\caption{Merging paths for $\bb = \tcb{1}001110011$ and $\bb = 01\tcb{1}101\tcb{1}011$.}
\label{fig:mergingpath1}
\end{figure}

We find that the sum of right lane lengths $R(B_\ell) = \sum_{\bb\in B_\ell} \rr(\bb)$ when summing over all arrival sequences of length $\ell$ results in the integer sequence \href{https://oeis.org/A230137}{A230137}:
\begin{equation} 0,2,6,18,44,110,252,588,1304, 2934, 6380, 14036, 30120, 65260, 138712,   \ldots \label{eq:sloane} \end{equation}

 Sloane's Online Encyclopedia of Integer Sequences notes that the sequence in \eqref{eq:sloane} divided by $2^\ell$, which we denote $\mathbb{E}[\ell] = R(B_\ell)/2^\ell$, is also the expected value of the maximum of the number of heads and the number of tails when $\ell$ fair coins are tossed \cite{oeis}. This fact suggests there is an explicit bijection between the set of merging paths of length $\ell$ and the sets of $\ell$ coin flips that sends the length of the right lane in a merging path to the maximum number of heads or tails in the corresponding sequence. 
 
 In Section \ref{section:bijection} we define a map with this property and prove it is indeed a bijection.  This bijection also gives a combinatorial proof of Theorem~\ref{theorem:closed}.
 
Section~\ref{section:kredcars} is dedicated to answering Question~\ref{question:2}, where we show that as $\ell$ tends to $\infty$,
$\mathbb{E}[\ell,k]/\ell $ tends to $\frac{1}{2}$ when $\ell \geq 2k$, and when $\ell < 2k$ the expected value $\mathbb{E}[\ell,k]/\ell $ tends to the ratio $k/\ell$.

In Section~\ref{section:domino} we explore a curious correspondence between the collection of arrival sequences with $1$ red car, longest domino snakes, and longest trails in complete graphs with loops.  We describe an explicit bijection that maps each arrival sequence to a subset of edges in a longest trail in the complete graph with loops or equivalently, a subset of dominoes in the longest domino snake with dominoes.

In Section~\ref{section:equiv_classes} we consider the final structure of an arrival sequence determined by the right lane vector that records the order of the cars in the right lane, disregarding their color. We then partition the collection of arrival sequences into color-blind equivalence classes based on their right lane vector, show that each color-blind equivalence class has even number of elements, and moreover, that the number of elements in each class is a power of 2.
We end this paper with a list of open problems and future directions.

\section{Merging Paths}\label{sec:mergingpaths}
We record the information of an arrival sequence with a decorated lattice path where right steps represent green cars (1s), up steps represent either red cars (0s), and decorated upward steps (or bounces) represent green cars choosing the right lane when the lanes are even. (Recall that when a green car is forced to choose the right lane, we call this a ``bounce", since the corresponding merging path bounces up off the line $y=x$ when it would normally head right.)
Let $M_n(m)$ be the number of such lattice paths reaching the point $(n,m)$, where the lattice path starts at the origin. Each of these \textit{merging paths} represents a sequence of cars that ends with $n$ cars in the right lane and $m$ cars in the left lane. As we remarked earlier, these paths never cross the diagonal $y=x$, so they are \textit{ballot paths}.  As an example, two merging paths are depicted above in Figure~\ref{fig:mergingpath1}, and places where they bounce off the diagonal are highlighted in blue.

The following table counts these paths for small values of $m$ and $n$, and the subsequent lemma gives a recurrence relation for these numbers.

 \begin{table}[h!]
\begin{tabular}{l||llllllllll}
$m$ & 2 & 16 & 72 & 240 &660  & 1584 & 3432 &3432  &    \\ 
$6$ & 2 & 14 & 56 & 168 & 420 & 924 & 924 &  &    \\ 
$5$ & 2 & 12 & 42 & 112 & 252 & 252 &  &  &    \\ 
$4$ & 2 & 10 & 30 & 70 & 70 &  &  &  &    \\ 
$3$ & 2 & 8 & 20 & 20 &  &  &  &  &    \\ 
$2$ & 2 & 6 & 6 &  &  &  &  &  &   \\ 
$1$ & 2 & 2 &  &  &  &  &  &  &    \\ 
$0$ & 1 &  &  &  &  &  &  &  &   \\ \hline\hline
& $0$ & $1$ & $2$ & $3$ & $4$ & $5$ & $6$ & $7$ & $n$ %
\end{tabular} \caption{Number of merging paths ending at $(n,m)$} \label{table:Mnm}
\end{table}

\begin{lemma}\label{lemma:recurrence}
The numbers $M_n(m)$ satisfy the following recurrence relation.
$$\begin{array}{ccll}
M_n(m)&=&M_{n-1}(m)+M_n(m-1) & \textnormal{ for } m>n+1, n>0,  \\
M_n(m)&=&M_{n-1}(m)+2M_n(m-1) & \textnormal{ for } m=n+1, n>0, \\
M_n(n)&=&M_{n-1}(n) & \textnormal{ for } n>0  \\
M_0(0)&=&1  \textnormal{ and } M_0(m)=2 & \textnormal{ for } m>0 \\
\end{array}$$
\end{lemma}

\begin{proof}
We prove the recurrence relation by induction on $n$. If $n=0$ then the only merging paths reaching $(0,m)$ are $00\ldots0$ and $10\cdots0$. Thus, $M_0(m)=2$ for $m>0$. The empty path is the only path reaching $(0,0)$, thus $M_0(0)=1$

Now suppose the recurrence relation is true for $n<\k$ and consider paths reaching $(\k,m)$. We now start a second induction argument on $m\geq \k$. If $m=\k$, then the only paths reaching $(\k,\k)$ come from paths reaching $(\k-1,\k)$ by appending a 1; thus

\[
M_\k(\k)=M_{\k-1}(\k).
\] If $m=\k+1$, then the paths reaching $(\k,\k+1)$ either come from paths reaching $(\k-1,\k+1)$ by appending a 1, or from paths reaching $(\k,\k)$ by appending either a 0 or a 1, since both would result in an up step in this case. Thus,

$$M_\k(\k+1)=M_{\k-1}(\k+1)+2M_\k(\k).$$

Finally, if $m>\k+1$, then the paths reaching $(\k,m)$ either come from paths reaching $(\k-1,m)$ by appending a 1, or from paths reaching $(\k,m-1)$ by appending a 0. Thus, 

\[
M_\k(m)=M_{\k-1}(m)+M_\k(m),
\] completing both induction arguments.
\end{proof}

Notice that ``folding" Pascal's triangle in half, i.e.~doubling the off-center values, gives the values in Table \ref{table:Mnm}. Thus, we have the following theorem that provides a closed formula for the numbers $M_n(m)$.

\begin{theorem}\label{theorem:closed}
The numbers $M_n(m)$ have the following closed formulas:

\begin{align}
M_n(m)&=2\dbinom{m+n}{n} \text{ for } m>n, \text{ and } \label{eqn:Mn1}
\\
M_n(n)&=\dbinom{2n}{n}. \label{eqn:Mn2}
\end{align}
\end{theorem}

\begin{proof}
We prove that equations (\ref{eqn:Mn1}) and (\ref{eqn:Mn2}) satisfy the recurrence relation in Lemma \ref{lemma:recurrence}.
First, notice that if $n=0$, then $2\binom{m+0}{0}=2$ for $m>0$ and $\binom{2\cdot0}{0}=1$.

From equation (\ref{eqn:Mn1}), we use the identity \[2\binom{2n+1}{n}=2\binom{2n}{n-1}+2\binom{2n}{n}\] to conclude that
$M_n(m)=M_{n-1}(m)+2M_n(m-1)$ for $m=n+1$ and $n>0.$
Next, using the identity \[\binom{2n}{n}=2\binom{2n-1}{n-1}\] for $n>0$, we get that
$M_n(n)=M_{n-1}(n)$ for $n>0.$
\end{proof}

One may reasonably ask for a combinatorial proof of Theorem~\ref{theorem:closed}, and in fact, one will be provided in Corollary \ref{corollary:combinatorial}.

We can now consider the expected length of the right lane for an arrival sequence of $\ell$ cars, which we denote  $\mathbb{E}[\ell] = R(B_\ell)/2^{\ell}. $ The following theorem gives a closed formula for $\mathbb{E}[\ell]$. 

\begin{theorem} \label{thm:E_l}
Let $\ell\in\mathbb{N}$. If $\ell$ is odd, then

\[
\mathbb{E}[\ell]=\dfrac{\ell}{2^\ell}\left(2^{\ell-1}+\dbinom{\ell-1}{(\ell-1)/2}\right).
\]
If $\ell$ is even, then

\[
\mathbb{E}[\ell]=\dfrac{\ell}{2^{\ell+1}}\left(2^\ell+\dbinom{\ell}{\ell/2}\right).
\]
\end{theorem}

\begin{proof}
Suppose $\ell$ is odd, then a path of length $\ell$ does not end on the diagonal. There are $M_{\ell-i}(i)=2\binom{\ell}{i} $ arrival sequences with right lane length $i$, so

\begin{align*} 2^\ell\mathbb{E}[\ell] & =  R(B_\ell)   = \sum_{\bb \in B_\ell} r(\bb) \\ &  =2\sum\limits_{i=(\ell+1)/2}^\ell i\dbinom{\ell}{i}  \\ &  =2\ell\sum\limits_{i=(\ell-1)/2}^{\ell-1}\dbinom{\ell-1}{i} \\
& =\ell\left(2^{\ell-1}+\dbinom{\ell-1}{(\ell-1)/2}\right).
\end{align*} A similar argument shows the case when $\ell$ is even. 
\end{proof}

We can simplify this result as the number of cars grows large by using Stirling's approximation $n! \sim \sqrt{2\pi n}\left( {\frac{n}{e}}\right)^n $ \cite{M08} to derive an approximation for the central binomial coefficient $\binom{2n}{n}\sim\frac{2^{2n}}{\sqrt{n\pi}}$, which results in the following corollary.  

\begin{corollary}\label{corollary:limit}
\[
\lim\limits_{\ell\rightarrow\infty}\dfrac{\mathbb{E}[\ell]}{\ell}=\dfrac{1}{2}
\]
\end{corollary} This means that for large $\ell$, the lanes tend to even out, and the effect of the bouncing is not very large. We will see in Section \ref{section:kredcars} that this limit will change, depending on the ratio of red cars to green cars.  

 \section{A length-preserving bijection between merging paths and coin flips} \label{section:bijection}
 Let $\Bn$ denote the set of $\ell$-bit binary strings that represent arrival sequences.
For each $\bb \in \Bn$, let $ \rr(\bb)$ denote the length of the right-hand lane.  
Let $z(\bb)=k$ denote the number of zeros in $\bb$ and $o(\bb)=\ell-z(\bb)$ denote the number of ones in $\bb$. Let $\Cn$ denote the collection of sequences of $\ell$ coin flips. For each $\cc \in \Cn$, let $h(\cc)$ denote the number of heads $H$ in $\cc$, $t(\cc)$ denote the number of tails $T$ in $\cc$, and let $\M(\cc):=\max\{h(\cc),t(\cc)\}$.

In this section we define
a function $\phi:\Bn\rightarrow \Cn$ that satisfies $\rr(\bb)=\M(\cc)$ whenever $\phi(\bb)=\cc$; that is, $\phi$ sends each arrival sequence $\bb$ with right lane length $\rr(\bb)$ to a sequence of coin flips $\cc$ whose maximum number $\M(\cc)$ of heads or tails equals $
\rr(\bb)$.  Then we show that $\phi$ is in fact a bijection.   
Before proceeding, we give a small example in Table~\ref{table:n4} to help illustrate how the map is defined. To differentiate the strings in $\Bn$ and $\Cn$ we label $0$ in $\cc$ as $H$ and $1$ in $\cc$ as $T$.

\begin{table}[h!]
\begin{center}
\begin{tabular}{|c|c|c|c|c|} \hline
$\bb$  & $\rr(\bb)$  & $\pp$  & $\cc$ & $\M(\cc)$ \\ \hline
0000 & 4& 0000 & HHHH & 4 \\
0100 & 3 & 0000 & HTHH & 3 \\
0010 & 3& 0000 & HHTH & 3\\
0001 & 3& 0000 & HHHT & 3\\
0011 & 2 & 0000 & HHTT & 2 \\
01\tcb{1}0 & 3 & 00\tcb{0}1 & HT\tcb{T}T& 3 \\
0101 & 2 & 0000 & HTHT & 2\\
01\tcb{1}1 & 2  & 00\tcb{0}1 & HT\tcb{T}H & 2\\
\tcb{1}000 & 4  &\tcb{0}111 & \tcb{T}TTT & 4 \\
\tcb{1}100 & 3 & \tcb{0}111 & \tcb{T}HTT & 3 \\
\tcb{1}010 & 3 & \tcb{0}111 & \tcb{T}THT & 3\\
\tcb{1}001 & 3 & \tcb{0}111 & \tcb{T}TTH & 3\\
\tcb{1}011 & 2  & \tcb{0}111 & \tcb{T}THH & 2 \\
\tcb{1}1\tcb{1}0 & 3  & \tcb{0}1\tcb{1}0 & \tcb{T}H\tcb{H}H& 3 \\
\tcb{1}101 & 2  & \tcb{0}111 & \tcb{T}HTH & 2\\
\tcb{1}1\tcb{1}1 & 2  & \tcb{0}1\tcb{1}0 & \tcb{T}H\tcb{H}T & 2\\ \hline
\end{tabular}
\end{center}
\caption{An example of the length preserving bijection given by $\phi$ }  \label{table:n4}
\end{table}

In order to define the function $\phi:\Bn\rightarrow \Cn$ we define a vector $\pp$ that records where the merging path of $\bb$ bounces off the diagonal. The vector $\pp$ will work like a light switch, being toggled on to $1$ directly after a bounce occurs and staying that way until another bounce toggles it back to 0.

\begin{definition}
The \textbf{parity vector} $\pp(\bb)$ (or  $\pp=p_1p_2 \cdots p_\ell$ if $\bb$ is understood) is defined from $\bb$ as follows:
\begin{itemize}
    \item The first entry in $\pp$ is always zero $p_1 =0$.
    \item  If the merging path of $\bb$ bounces off the diagonal with $b_i=1$ then $p_{i+1} = \overline{p_{i}}$.
    \item Otherwise $p_{i+1} = p_{i}$ for all $ 1 \leq i \leq \ell-1$. 
\end{itemize}

\end{definition}

  In terms of the arrival sequences, the parity of $\pp$ changes one step after a car that is labeled with a $1$ goes into the right lane.  In terms of lattice paths, the parity vector changes parity after the path bounces off the diagonal with a $1$.

Here are a few examples to help clarify these definitions:
\begin{itemize}
    \item If 
    $\bb=01\tcb{1}101\tcb{1}1$,   then
    $\pp =00\tcb{0}111\tcb{1}0$, and 
    $\cc =01\tcb{1}010\tcb{0}1 $ (or HT\tcb{T}HTH\tcb{H}T). 
    \item  If 
    $\bb=\tcb{1}001110011$, then
    $\pp=\tcb{0}111111111$, and 
    $\cc=\tcb{1}110001100$.
 \item If 
 $\bb=\tcb{1}01011\tcb{1}0$, then 
 $\pp=\tcb{0}11111\tcb{1}0$,  and
 $\cc=\tcb{1}10100\tcb{0}0$.
\end{itemize}

We are ready to define the function $\phi$ using the parity vector $\pp$.  The map $\phi:\Bn \rightarrow \Cn$ is given by  $\phi(\bb)=\bb+\pp = \cc$ for all $\bb \in \Bn$.

\begin{remark}
We also note that the parity vector $\pp$ can easily be defined from $\cc$ as well as $\bb$.  The parity vector $\pp$ starts with $0$, and its parity changes the first time that the number of tails ($1$s) in $\cc$ outnumbers the number of heads ($0$s), then the parity changes back when the number of heads ($0$s) outnumbers the number of tails ($1$s) and so on.
\end{remark}

\begin{proposition}\label{prop:onetoone}
The function $\phi: \Bn \rightarrow \Cn$ is one-to-one.
\end{proposition}

\begin{proof}
Suppose $\bb,\bb'\in B$ are distinct elements of $\Bn$. Let $i$ be the minimum index for which $b_i\neq b_i'$. Then by definition of $\pp$, $p_i=p_i'$ and so $b_i+p_i \not\equiv b_i'+p_i'\pmod{2}$. Therefore, $\phi(\bb)\neq\phi(\bb')$, and we conclude that $\phi$ is one-to-one.
\end{proof}

We will soon show that $r(\bb) = \M(\phi(\bb))$ in Theorem~\ref{theorem:bijection}, and the following lemma provides the base case for the induction argument in that proof.

\begin{lemma} \label{lemma:basecase}
Suppose $\bb=b_1b_2\cdots b_\ell$ is an arrival sequence whose merging path leaves the diagonal and never has a bounce in positions $b_2 \cdots  b_\ell$, i.e.~its parity vector is constant from $p_2 \cdots p_\ell$. If $\phi(\bb)=\cc$, then $\M(\cc) = \rr(\bb)$.  
\end{lemma}

We have defined $\pp$ so it always starts with a zero, but for the sake of completeness, and because it will prove useful in our induction arguments, we will prove this lemma for all possible cases where $\pp$ starts with $0$ or $1$.

\begin{proof}
There are four cases to consider:
  $\pp=00 \cdots 0$, $\pp=01\cdots 1$, $\pp=10\cdots 0$, and $\pp=11\cdots 1$. First, if $\pp=00 \cdots 0$, then $\rr(\bb) = z(\bb)$ because otherwise there would be more $1$s than $0$s in $\bb$ which would force a $1$ to be in the right lane, contradicting the assumption that $\pp=00\cdots 0$.  Additionally, $\bb + \pp = \bb =\cc$ and $\rr(\bb) = z(\bb) = h(\cc) = \M(\cc)$.

If $\pp=01\cdots1$ then $\bb$ bounces off the diagonal at $b_1=1$ and nowhere else.  In this case, $\rr(\bb)= z(\bb)+1$ is the number of zeros in $\bb$ plus $1$ because that first entry $b_1=1$ ended up in the right lane, but all other $1$s will choose the left lane.    Since $\pp=01\cdots 1$, we have $t(\cc) = z(\bb)+1$ because $b_1+p_1=1+0=1$, and $b_i+p_i= \overline{b_i}$ for $2 \leq i \leq n$ so the first $1$ in $\bb$ corresponds to a tail in $\cc$, all other $1$s in $\bb$ correspond to heads in $\cc$, and every zero in $\bb$ corresponds to a tail in $\cc$. Finally, since $\bb$ never bounces back off the diagonal $\rr(\bb)=z(\bb)+1=t(\cc)=\M(\cc)$.  

 If $\pp=10\cdots 0$ then $\bb$ bounces off the diagonal at $b_1 = 1$ and nowhere else.  Hence $\cc= \overline{b_1}b_2 \cdots b_\ell$. Then $\rr(\bb) = z(\bb)+1$ because if $o(b_1\cdots b_j)>z(b_1b_2\cdots b_j)+1$ for any $2 \leq j \leq \ell$ then $\pp$ would switch its parity more than once. Hence $\rr(\bb) = z(\bb)+1 = h(\cc) =\M(\cc)$ because $\cc=\overline{b_1}b_2\cdots b_\ell = 0b_2\cdots b_\ell$ and there is one more zero/head in $\cc$ than zeros in $\bb.$

 Lastly, if $\pp=11\cdots 1$  then $\rr(\bb) = z(\bb)$ because otherwise there would be more $1$s than $0$s in $\bb$ which would force a $1$ to be in the right lane, which would contradict the assumption that $\pp=11\cdots 1$.  Additionally $\cc = \bb + \pp = \overline{\bb}$ so  $ \rr(\bb) = z(\bb) = t(\cc)= \M(\cc)$. 
\end{proof}

We note that the previous lemma does allow for the merging path to touch the diagonal again after $b_1$, but it does not allow the merging path to bounce off the diagonal anywhere in $b_2 \cdots b_\ell$. For instance, Figure~\ref{fig:latticepath1} shows two such paths. The first path corresponds to $\bb=\tcb{1}001110011$ where
$\pp=\tcb{0}111111111$ and 
$\cc=\tcb{1}110001100$.  Note that since $b_1=1$ in this example $\rr(\bb) = z(\bb)+1=5=t(\cc)=\M(\cc)$.
The second path corresponds to $\bb=0001110011$, for which $\pp=0000000000$ and 
$\cc=0001110011$.  Since $b_1=0$ in this example, $\rr(\bb) = z(\bb)=5=h(\cc)=\M(\cc)$.

\begin{figure}[h!]
\begin{tikzpicture}
    \NEpath{0,0}{5}{5}{0,0,0,1,1,1,0,0,1,1};
    \draw[blue, line width=2pt,-stealth] (0,0) -- ++(0,1);
    \end{tikzpicture} \hspace{12pt}
\begin{tikzpicture}
    \NEpath{0,0}{5}{5}{0,0,0,1,1,1,0,0,1,1};
    \end{tikzpicture}
\caption{Two merging paths satisfying the hypotheses of Lemma~\ref{lemma:basecase}.}
\label{fig:latticepath1}
\end{figure}

The next lemma shows that if any lattice path corresponding to a merging path $\bb=b_1\cdots b_\ell$ ends on the diagonal, then $\rr(\bb) = \M(\phi(\bb)) = \frac{\ell}{2}$.

\begin{lemma} \label{lemma:diagonal}
If $\bb=b_1b_2\cdots b_\ell$ is an arrival sequence whose merging path ends on the diagonal, then $\rr(\bb) =\M(\cc) = \frac{\ell}{2}.$
\end{lemma}

\begin{proof}
We prove the claim by induction on the number of times the merging path bounces off the diagonal with a $1$ after originally leaving the diagonal (or equivalently, the number of times the parity vector changes its parity in $p_2\cdots p_\ell$).
The base case was proven in Lemma~\ref{lemma:basecase}. (To reiterate, 
if $b_1=0$, then $z(\bb) = h(\cc) = t(\cc) =\M(\cc) = \frac{\ell}{2}$.
If $b_1=1$, then $z(\bb) = o(\bb)-2$, and $\rr(\bb)=z(\bb)+1=t(\cc)=o(\bb)-1=h(\cc)=\M(\cc) = \frac{\ell}{2}$.)  

Now suppose that the claim holds if the merging path bounces returns and bounces off the diagonal $i$ times.  Let $\bb$ denote any arrival sequence whose merging path returns and bounces off the diagonal a total of $i+1$ times, and let $j$ be the index where the merging path of $\bb$ bounces of the diagonal for the last time. Then by induction we know $$r(b_1\cdots b_{j})=t(c_1 \cdots c_{j})=h(c_1 \cdots c_{j})=\M(c_1 \cdots c_{j}) =\frac{j}{2}$$  and $b_{j+1} \cdots b_\ell$ corresponds to a merging path that satisfies the hypotheses of Lemma~\ref{lemma:basecase}.  So
$r(b_{j+1}\cdots b_{\ell})=t(c_{j+1} \cdots c_{\ell})=h(c_{j+1} \cdots c_{\ell})=\M(c_{j+1} \cdots c_{\ell}) =\frac{\ell-j}{2}$.
Hence,
\begin{align*}
    \rr(\bb) &= \rr(b_1 \cdots b_{j}) + \rr(b_{j+1} \cdots b_\ell) \\
    &=   \frac{j}{2} + \frac{\ell-j}{2}  \\
    &= \M(c_1 \cdots c_{j}) + \M(c_{j+1} \cdots c_\ell)  \\
    &=\M(\cc).
\end{align*}  By the principal of mathematical induction the claim holds in general.
\end{proof}

The lattice path depicted in Figure~\ref{fig:latticepath2} shows a path that satisfies the hypotheses of Lemma~\ref{lemma:diagonal}. That path corresponds to $\bb=010011\tcb{1}011\tcb{1}1$ and bounces off the diagonal twice. We note that $\pp=000000\tcb{0}111\tcb{1}0$ and $\cc=010011\tcb{1}100\tcb{0}1$, so $\rr(\bb)=6=z(\bb)+2$, $z(\bb)=4$, and $\M(\cc) =6$.

\begin{figure}[h!]
\begin{tikzpicture}
    \NEpath{0,0}{6}{6}{0,1,0,0,1,1,0,0,1,1,0,1};
    \draw[blue, line width=2pt,-stealth] (3,3) -- ++(0,1);
    \draw[blue, line width=2pt,-stealth] (5,5) -- ++(0,1);
    \end{tikzpicture} 
    \caption{A merging path satisfying the hypotheses of Lemma~\ref{lemma:diagonal}.  } \label{fig:latticepath2}
\end{figure}
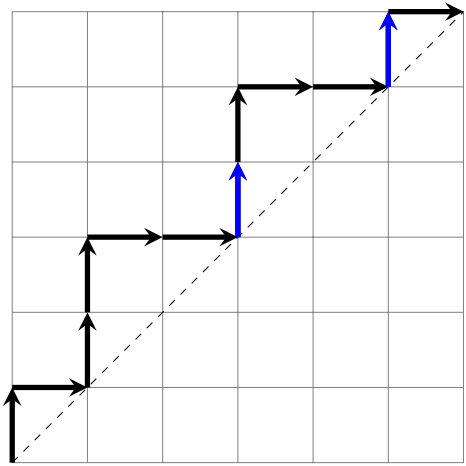

We are now ready to state the main result of this section.
\begin{theorem}\label{theorem:bijection}
The function $\phi: \Bn \rightarrow \Cn$ takes a string $\bb$ with right lane length $\rr(\bb)$ to a sequence of coin flips $\phi(\bb) =\cc$ with a max number of heads or tails satisfying $\M(\cc)=\rr(\bb)$.
\end{theorem}

\begin{proof}
Let $\bb$ be any arrival sequence in $\Bn$ and $\cc=\phi(\bb)$. Let $j$ denote the last index where  the merging path defined by $\bb$ touches the diagonal.
  Then by Lemma~\ref{lemma:diagonal} we know that $r(b_1\cdots b_j)=\M(c_1\cdots c_j) = \frac{j}{2}$.  
Moreover, the path corresponding to $b_{j+1}\cdots b_\ell$ satisfies the hypotheses of Lemma~\ref{lemma:basecase} so $r(b_{j+1}\cdots b_\ell) = \M(c_{j+1} \cdots c_\ell)$.
We conclude that \[ \rr(\bb)= r(b_1 \cdots b_j) + r(b_{j+1} \cdots b_\ell) = \M(c_1 \cdots c_j) + \M (c_{j+1} \cdots c_\ell) =\M (\cc).  \qedhere \]
\end{proof}

 Figure~\ref{fig:latticepath3} illustrates how we break a merging path into two parts. The first part is the longest subpath $b_1\cdots b_j$ that ends on the diagonal and then bounces off with $b_{j+1}=1$, and the second part is the remaining subpath $b_{j+1}\cdots b_\ell$ that satisfies the hypotheses of Lemma~\ref{lemma:basecase}. This specific merging path corresponds to $\bb=010011\tcb{1}011\tcb{1}001$, and in this case $j=10$. The subpath $b_1 \cdots b_{j}=010011\tcb{1}011$ ends up on the diagonal at $(5,5)$.
The path $\tcb{b_{j+1}}b_{j+2}\cdots b_\ell= \tcb{1}001$ has $r(\tcb{b_{j+1}}b_{j+2}\cdots b_\ell) = 3 = z(\tcb{b_{j+1}}b_{j+2}\cdots b_\ell)+1$.  
For this particular $\bb=010011\tcb{1}011\tcb{1}001$, one can confirm that the parity vector is $\pp=000000\tcb{0}111\tcb{1}000$ and $\phi(\bb)=\bb+\pp=\cc=010011\tcb{1}100\tcb{0}001$. 
Moreover, $\rr(\bb) = r(b_1\cdots b_j) + r(b_{j+1}\cdots b_{\ell})=5+3 = \M(c_1\cdots c_j) + \M(c_{j+1} \cdots c_\ell) =\M(\cc) .$

\begin{figure}[h!]
\begin{tikzpicture}
    \NEpath{0,0}{8}{8}{0,1,0,0,1,1,0,0,1,1,0,0,0,1};
    \draw[blue, line width=2pt,-stealth] (3,3) -- ++(0,1);
    \draw[blue, line width=2pt,-stealth] (5,5) -- ++(0,1);
    \end{tikzpicture} 
\caption{An example of a general merging path.} \label{fig:latticepath3}
\end{figure}
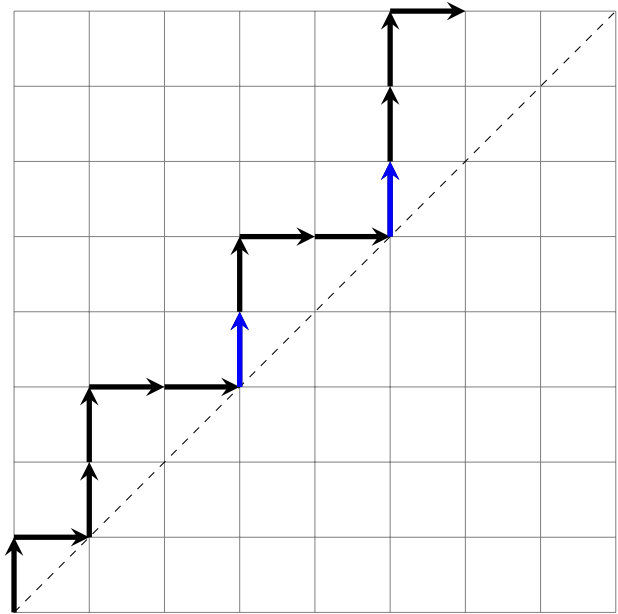

We use Theorem \ref{theorem:bijection} to give a combinatorial proof of Theorem \ref{theorem:closed} in the following corollary.

\begin{corollary}\label{corollary:combinatorial}
The function $\phi$ is a bijection. That is, the merging paths reaching the point $(n,m)$ are in one-to-one correspondence with heads/tails sequences $\cc$ with $m=\M(\cc)$.
\end{corollary}

\begin{proof}
Let $m=\M(\cc)$, where $\cc\in C_{\ell}$, and let $n=\ell-m$. Then the number of coin flips with max $m$ of heads or tails is $2\binom{m+n}{m}$, when $m>n$, and $\binom{2n}{n}$ when $m=n$. Since these match the formulas given in Theorem \ref{theorem:closed} and $\phi$ is one-to-one by Proposition \ref{prop:onetoone}, $\phi$ must be a bijection. Thus, $\phi$ provides a combinatorial proof of Theorem \ref{theorem:closed}.
\end{proof}

\section{Merging Paths with exactly $k$ red cars}\label{section:kredcars}

We now analyze Question~\ref{question:2}: what is the expected length of the right lane when we consider all possible arrival sequences with exactly $k$ red cars? The following table gives the numbers $R_{\ell,k}$, which is the total number of cars in the right lane when summed over all binary sequences of length $\ell$ and exactly $k$ zeros.

 \begin{table}[h!]
\begin{tabular}{l}
$%
\begin{tabular}{l||llllllllll}
$\ell$ & 4 & 33 & 120 & 253 & 344 & 309 & 176 & 57 & 8 &  \\ 
$7$ & 4 & 28 & 85 & 147 & 162 & 112 & 43 & 7 &  &  \\ 
$6$ & 3 & 19 & 51 & 76 & 66 & 31 & 6 &  &  &  \\ 
$5$ & 3 & 15 & 31 & 35 & 21 & 5 &  &  &  &  \\ 
$4$ & 2 & 9 & 16 & 13 & 4 &  &  &  &  &  \\ 
$3$ & 2 & 6 & 7 & 3 &  &  &  &  &  &  \\ 
$2$ & 1 & 3 & 2 &  &  &  &  &  &  &  \\ 
$1$ & 1 & 1 &  &  &  &  &  &  &  &  \\ 
$0$ & 0 &  &  &  &  &  &  &  &  &  \\ \hline\hline
& $0$ & $1$ & $2$ & $3$ & $4$ & $5$ & $6$ & $7$ & $8$ & $k$%
\end{tabular}%
$ \\ 
\end{tabular}%
\caption{Sum of right lane lengths for all merging paths of length $\ell$ with $k$ zeros} \label{table:Rlk}
\end{table}

To calculate the above numbers, we consider merging paths with exactly $k$ red cars. Let $\Wmnk$ denote the set of all merging paths reaching the point $(n,m)$ with exactly $k$ zeros, and $M_{n,k}(m)=|\Wmnk|$ be the number of such merging paths.

Table~\ref{table:mnk} counts these paths for small values of $m$, $n$, and $k$, and the following lemma describes some recursive formulas for the values $M_{n,k}(m)$.  The proof is straightforward and is essentially the same as that of  Lemma~\ref{lemma:recurrence} so we omit it.

\begin{lemma} \label{prop:mnk_recursions}
The values $M_{n,k}(m)$ satisfy the following recursive formulas:
\begin{itemize}
\item $M_{n,k}(m)=M_{n-1,k}(m)+M_{n,k-1}(m-1)$ for $m>n+1$, $n>0$,  \\
\item $M_{n,k}(m)=M_{n-1,k}(m)+M_{n,k-1}(m-1)+M_{n,k}(m-1)$ for $m=n+1$, $n>0$, \\
\item $M_{n,k}(n)=M_{n-1,k}(n)$ for $n>0$, and  \\
\item $M_{0,k}(m)=1$ when $m=k$ or $k+1$, and $k>1$; and when $k=0$ and $m=1$.\\
\item Otherwise, $M_{0,k}(m)=0$ .
\end{itemize}
\end{lemma}

\begin{table}[h]
\begin{tabular}{l||llllll||llllll||llllll||llllll}
$k$ & $0$ &  &  &  & &  & $1$ &  &  &  &  &  &$2$ &  &  &   &  &  & $3$ &  &  &  &  &   \\ \hline
$m$ & 0 & 0 & 0 & 0 & 0 & 1 & 0 & 0 & 0 & 0 & 1 & 11 & 0 & 0 & 0 & 1 & 10 & 54 & 0 & 0 & 1 & 9 & 44 & 154  \\ 
$5$ & 0 & 0 & 0 & 0 & 1 & 1 & 0 & 0 & 0 & 1 & 9 & 9  & 0 & 0 & 1 & 8 & 35 & 35 & 0 & 1 & 7 & \textcolor{green}{27} & \textcolor{green}{75} & 75    \\ 
$4$ & 0 & 0 & 0 & 1 & 1 &   & 0 & 0 & 1 & 7 & 7 &    & 0 & \textcolor{blue}{1} & \textcolor{blue}{6} & 20 & \textcolor{green}{20}&    & 1 & 5 & 14 & 28 & \textcolor{green}{28} &  \\ 
$3$ & 0 & 0 & 1 & 1 &   &   & 0 & 1 & \textcolor{blue}{5} & 5 &   &    & 1 & 4 & 9 & 9 &   &     & 1 & 3 & 5 & 5 &   & \\ 
$2$ & 0 & 1 & 1 &   &   &   & 1 & 3 & 3 &   &   &    & 1 & 2 & 2 &   &   &     & 0 & 0 & 0 &   &   & \\ 
$1$ & 1 & 1 &   &   &   &   & 1 & 1 &   &   &   &    & 0 & 0 &   &   &   &     & 0 & 0 &   &   &   &  \\ 
$0$ & 0 &   &   &   &   &   & 0 &   &   &   &   &    & 0 &   &   &   &   &     & 0 &   &   &   &   &      \\ \hline\hline
 $n$ &    $0$&$1$&$2$&$3$&$4$&$5$&$0$&$1$&$2$&$3$&$4$&$5$ &$0$&$1$&$2$&$3$&$4$&$5$&$0$&$1$&$2$&$3$&$4$&$5$%
\end{tabular}
\caption{Values of $M_{n,k}(m)$ for small values of $m,n,$ and $k$. } \label{table:mnk}
\end{table}

Let $\Tlbk$ denote the set of all arrival sequences that have length $\ell$, contain exactly $k$ zeros, and contain at least $b$ bounces. The following proposition gives the connections between $\Wmnk$ and $\Tlbk$.

\begin{proposition}\label{prop:WandT}
$$\Tmnk = \bigcup_{i=0}^n  \mathbf{W}_{n-i,m+i,k} \quad \text{ and } \quad |\Tmnk| = \sum_{i=0}^n  M_{n-i,k}(m+i).$$

$$\Wmnk = \Tmnk-\mathbf{T}_{m+n,m-k+1,k}$$
\end{proposition}

\begin{proof}
For the first formula, since each sequence in $\Tmnk$ has at least $m-k$ bounces and exactly $k$ zeros, each merging path takes at least $m$ steps up. Thus each such path ends weakly northwest of $(n,m)$. 
The second formula follows since the number of merging paths with $k$ zeros that end up weakly northwest of $(n-1,m+1)$ is \[\mathbf{T}_{m+1+n-1,m+1-k,k}=\mathbf{T}_{m+n,m-k+1,k}. \qedhere\]
\end{proof}

Let $B_{\ell,k}$ denote the set of binary sequences of length $\ell$ with exactly $k$ zeros.  The following lemma gives bounds on the number of bounces in a merging path with exactly $k$ zeros.

\begin{lemma} \label{lemma:bounces}
Let $\bb \in B_{\ell,k}$ be a binary string of length $\ell$ with $k$ zeros, and whose merging path contains $b$ bounces. Then
\[
\dfrac{\ell-2k}{2}\leq b\leq \dfrac{\ell-k+1}{2}
\]
\end{lemma}

\begin{proof}
Suppose the merging path reaches the point $(n,m)$. Then $m=k+b$ and $n=\ell-k-b$. The lower bound results from the fact that $m\geq n$. For the upper bound, we note that, with the exception of an initial bounce, each bounce is preceded by a 1. Thus

\begin{equation*}
b\leq\dfrac{o(\bb)+1}{2}=\dfrac{l-k+1}{2}. \qedhere    
\end{equation*}
\end{proof}

Now we are ready to prove the connection between the sets $\Tmnk$ and $B_{m+n,n-(m-k)+1}$ through an argument similar to Andre's Reflection Method \cite{A87}.  Given an arrival sequence $\bb$, let $\overline{\bb}$ denote the sequence satisfying $\overline{b_i} = (b_i+1) \mod 2$ for all $1 \leq i \leq m+n$. If $m>k$, we can write each arrival sequence $\bb \in \Tmnk$ as $\bb=b_1b_2$ where the last element of $b_1$ is where the $(m-k)$th bounce off the diagonal occurs in the corresponding merging path. We can define a map $\psi: \Tmnk \rightarrow B_{m+n,n-(m-k)+1}$ that sends each $\bb=b_1b_2$ to $\bb=b_1 \overline{b_2}$. The following Lemma shows that $\psi$ is a bijection, and gives us a formula for the numbers $\Tmnk$.

\begin{lemma} \label{lemma:psi_bij}
If $m>k$ and $m>n$, then the map $\psi: \Tmnk \rightarrow B_{m+n,n-(m-k)+1}$ that sends each $\bb=b_1b_2$ to $\bb'=b_1 \overline{b_2}$ is a bijection. Hence  \[|\Tmnk|= \dbinom{m+n}{n-(m-k)+1}.\]
\end{lemma}

\begin{proof}
We start by showing that $\psi$ is well-defined; that is, $\psi$ sends each arrival sequence $\bb = b_1b_2$ in $\Tmnk$ to an arrival sequence $\bb'=b_1\overline{b_2}$ in $B_{m+n,n-(m-k)+1}$.  Let $\bb=b_1b_2 \in \Tmnk$ be an arrival sequence of length $m+n$, containing exactly $k$ zeros, and whose merging path contains at least $m-k$ bounces. Suppose $b_1$ contains $j$ zeros for some $0 \leq j \leq k$. In order for the last entry in $b_1$ to be the $(m-k)$th bounce in $b_1$, it must be the case that $b_1$ contains $2(m-k)-1+j$ ones. 
We calculate the number of ones in $b_2$ as
\begin{align*}
  o(b_2) & =  m+n-o(b_1)-z(\bb) \\ & = m+n-(2(m-k)-1+j)-k  \\
                             & = n-m+k+1-j.
\end{align*}
Now we calculate that the number of zeros in $\bb'=\psi(\bb)$ is $$z(\psi(\bb)) = z(b_1 \overline{b_2})  = z(b_1)+o(b_2) = j+(n-m+k+1-j) = n-(m-k)+1.$$
Hence $\psi$ sends each $\bb=b_1b_2 \in \Tmnk$ to a sequence $\bb' =b_1\overline{b_2}$ with exactly $n-(m-k)+1 $ zeros in $B_{m+n,n-(m-k)+1}$.

Next we show that $\psi$ is injective.  Suppose that $\bb^1$ and $\bb^2$ are distinct arrival sequences in $\Tmnk$.  Then either $b^1_1 \neq b^2_1$ or 
$b^1_2 \neq b^2_2$. In either case, we see that $\bb^1 \neq \bb^2$ implies that $$ \psi(\bb^1) = b^1_1\overline{b^1_2} \neq b^2_1\overline{b^2_2} =\psi(\bb^2). $$

We now argue that $\psi:\Tmnk\rightarrow B_{m+n,n-(m-k)+1}$ is surjective.
Let $\bb \in B_{m+n,n-(m-k)+1}$ be an arrival sequence with $n-(m-k)+1$ zeros. Consider the associated merging path and write $\bb=b_1b_2$ where the $(m-k)$th bounce is the last entry in $b_1$. Lemma~\ref{lemma:bounces} guarantees that the merging path for $\bb$ will have at least $m-k$ bounces when $m>n$, (we leave this as an exercise for the reader). Then $b_1 \overline{b_2}$ is an arrival sequence in $\Tmnk$ as it contains at least $m-k$ bounces, and we calculate that it has exactly $k$ zeros as follows:
\begin{itemize}
    \item $z(b_1)=j$
    \item $o(b_1)=2(m-k)-1+j$
    \item $o(b_2) = m+n-o(b_1)-z(\bb) = m+n-(2(m-k)-1+j)-(n-(m-k)+1) = k-j $
    \item $z(b_1 \overline{b_2}) = z(b_1)+o(b_2) = j+(k-j) =k$
\end{itemize}
Hence $\psi(b_1 \overline{b_2})=\bb$ and so $\psi$ is a surjective map. Therefore, $\psi$ is a bijection.
\end{proof}

The following theorem follows from Lemma \ref{lemma:psi_bij} and Proposition \ref{prop:WandT} and gives us one of the formulas for the merging paths $M_{n,k}(m)$.

\begin{theorem} \label{prop:general_M_mnk}
If $m > k$ and $m> n$, then
\[
M_{n,k}(m)=\dbinom{m+n}{n-(m-k)+1}-\dbinom{m+n}{n-(m-k)-1}.
\]
\end{theorem}

The previous results leave out merging paths that reach the diagonal where $m=n$, and merging paths where $m=k$. For the case when $m=n$, the corresponding set $\mathbf{T}_{2n,n-k,k}$ has size $\binom{2n}{k}$. The following corollary gives the formula for the merging paths reaching the diagonal $M_{n,k}(n)$.

\begin{corollary} \label{prop:diagonal}
If $n\geq k$, then
\[
M_{n,k}(n)=\dbinom{2n}{k}-\dbinom{2n}{k-1}.
\]
\end{corollary}

\begin{proof}
By Lemma \ref{prop:WandT} and Lemma \ref{lemma:psi_bij}, we have
\begin{equation*}
M_{n,k}(n)=\mathbf{W}_{n,n,k} = \mathbf{T}_{2n,n-k,k} - \mathbf{T}_{2n,n-k+1,k}=\dbinom{2n}{k}-\dbinom{2n}{k-1}. \qedhere
\end{equation*}
\end{proof}

The final case is when $m=k$ which corresponds to the set of merging paths with zero bounces. These paths are counted by the ballot numbers \cite{NS10}, so we record the formula for them as the following Theorem. The proof is the original Andre Reflection Method \cite{R08}.

\begin{theorem}\label{theorem:ballot}
\[
M_{n,k}(k)=\dfrac{k-n+1}{k+1}\dbinom{k+n}{n}=\dbinom{k+n}{n}-\dbinom{k+n}{n-1}.
\]
\end{theorem}

\begin{table}[h]
\begin{tabular}{l}
$%
\begin{tabular}{l||llllllllllll}
$m$ & 0 & 0 & 0 & 0 & 1 & 16 & 135 & 798 & 3705 & 14364 & 48279 & \textcolor{green}{48279} \\ 
$10$ & 0 & 0 & 0 & 1 & 14 & 104 & 544 & 2244 & 7752 & 23256 & \textcolor{green}{23256}\\ 
$9$ & 0 & 0 & 1 & 12 & 77 & 350 & 1260 & 3808 & 9996 & \textcolor{green}{9996} \\ 
$8$ & 0 & 1 & 10 & 54 & 208 & 637 & 1638 & 3640 & \textcolor{green}{3640} &  \\ 
$7$ & 1 & 8 & 35 & 110 & 275 & 572 & 1001 & \textcolor{green}{1001} &  & \\ 
$6$ & \textcolor{blue}{1} & \textcolor{blue}{6} & \textcolor{blue}{20} & \textcolor{blue}{48} & \textcolor{blue}{90} & \textcolor{blue}{132} & \textcolor{blue}{132} &  &  &  \\ \hline\hline
& $0$ & $1$ & $2$ & $3$ & $4$ & $5$ & $6$ & $7$ & $8$ & $9$ & $10$ & $n$%
\end{tabular}%
$ \\ 
\end{tabular}%
\caption{Values of $M_{n,k}(m)$ for $k=6$ where the numbers along the diagonal in green are given by Corollary \ref{prop:diagonal}, the ballot numbers along the bottom in blue are given by Theorem \ref{theorem:ballot}, and the remaining numbers are given by Theorem \ref{prop:general_M_mnk}. }
\end{table}

Many of the values of $M_{n,k}(m)$ repeat periodically. The following lemma shows where that repetition occurs, and gives a bijective proof. (The formulas above would give a trivial proof of this result.)

\begin{lemma} \label{lemma:step}
If $m > k+1$ and $m>n>0$, then the number of merging paths from $(0,0)$ to $(n,m)$ with exactly $k$ zeros is equal to the number of paths from $(0,0)$ to $(n-1,m+1)$ with exactly $k+2$ zeros. In other words,
$$M_{n,k}(m)=M_{n-1,k+2}(m+1).$$
\end{lemma}
\begin{proof}
Let $\bb\in\Wmnk$.  Since we assume that $m>k+1$, the number  of bounces in $\bb$ is $m-k>1$.  Hence there is at least one bounce not at the origin. The last of these bounces off the diagonal in $\bb$ comes from two consecutive 1s (the first of which occurs at one entry off the diagonal), changing those 1s to 0s creates a merging path $\bb' \in \mathbf{W}_{n-1,m+1,k+2}$ that ends at $(n-1,m+1)$ and has two more zeros than $\bb$. To reverse this process simply look for the last place a path $\bb'$ from $(0,0)$ to $(n-1,m+1)$ is distance one from the diagonal. There must follow two consecutive $0$s, so we replace those two entries with $1$s, giving us a path to $(n,m)$ with exactly $k$ zeros. 
\end{proof}

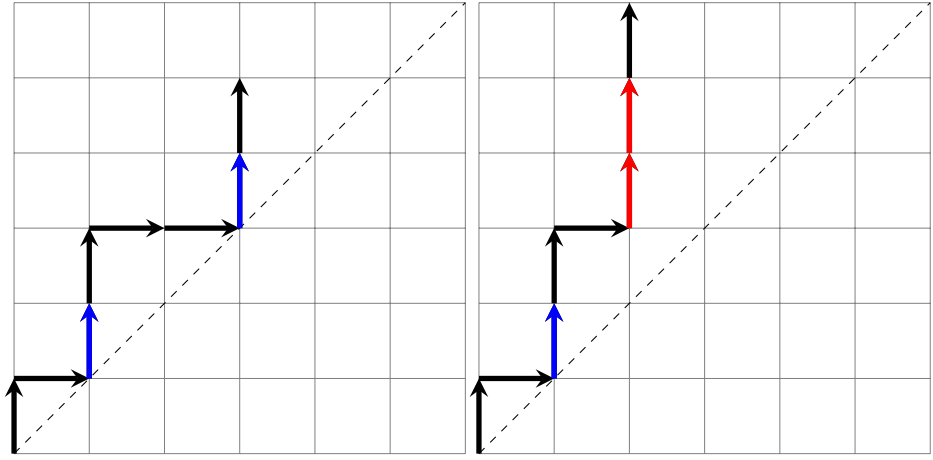
\begin{figure}[h!]
\begin{tikzpicture}
    \NEpath{0,0}{6}{6}{0,1,0,0,1,1,0,0};
    \draw[blue, line width=2pt,-stealth] (1,1) -- ++(0,1);
    \draw[blue, line width=2pt,-stealth] (3,3) -- ++(0,1);
    \end{tikzpicture} 
\begin{tikzpicture}
    \NEpath{0,0}{6}{6}{0,1,0,0,1,0,0,0};
    \draw[blue, line width=2pt,-stealth] (1,1) -- ++(0,1);
    \draw[red, line width=2pt,-stealth] (2,3) -- ++(0,1);
    \draw[red, line width=2pt,-stealth] (2,4) -- ++(0,1);
    \end{tikzpicture} 
\caption{A merging path 01\tcb{1}011\tcb{1}0 ending at (3,5) and a path 01\tcb{1}01\tcr{00}0 ending at (2,6).} \label{fig:latticepath4}
\end{figure}

Figure~\ref{fig:latticepath4} illustrates the bijection in Lemma~\ref{lemma:step}. It shows a path in $W_{3,3,5}$ and its corresponding path in $W_{2,5,7}$.

Using the formulas in Theorem \ref{prop:general_M_mnk}, Corollary \ref{prop:diagonal}, and Theorem \ref{theorem:ballot}, we can write down the formulas for the expected length of the right lane for $\ell$ cars with $k$ red cars. The formula breaks into 3 cases, as illustrated in the next theorem.

\begin{theorem}\label{theorem:expectedk}
The expected length $\mathbb{E}[\ell,k]$ of the right lane for $\ell$ cars with $k$ red cars is

\[
\mathbb{E}[\ell,k]=\left.\left[\dfrac{\ell+1}{2}\dbinom{\ell}{k}+\sum\limits_{i=0}^{k/2-1}\dbinom{\ell}{k-2i-2}\right]\middle/\dbinom{\ell}{k}\right.
\] for $\ell\geq 2k+1$ and $\ell$ odd,

\[
\mathbb{E}[\ell,k]=\left.\left[\dfrac{\ell}{2}\dbinom{\ell}{k}+\sum\limits_{i=0}^{(k-1)/2}\dbinom{\ell}{k-2i-1}\right]\middle/\dbinom{\ell}{k}\right.
\] for $\ell\geq 2k$ and $\ell$ even, and

\[
\mathbb{E}[\ell,k]=\left.\left[k\dbinom{\ell}{k}+\sum\limits_{i=0}^{(\ell-k-1)/2}\dbinom{\ell}{k+2i+1}\right]\middle/\dbinom{\ell}{k}\right.
\] for $\ell<2k$.
\end{theorem}

Note that if we let $k'=\ell-k$ in the last equation above, we get

\[
\mathbb{E}[\ell,k]=\left.\left[k\dbinom{\ell}{k}+\sum\limits_{i=0}^{(k'-1)/2}\dbinom{\ell}{k'-2i-1}\right]\middle/\dbinom{\ell}{k}\right.
\] for $\ell>2k'.$

\begin{proof}
We prove only the first case when $\ell\geq 2k+1$ and $\ell$ is odd, as the other cases are similar. When $\ell$ is odd, the minimum height of a merging path is $\frac{\ell+1}{2}$, and the maximum height by Lemma \ref{lemma:bounces} is $m=k+b=\frac{\ell+k+1}{2}$. So,

\begin{eqnarray}
\notag \dbinom{\ell}{k} \mathbb{E}[\ell,k]&=&\left[\sum\limits_{i=0}^{k/2}\left(\dfrac{\ell+1}{2}+i\right)M_{\frac{\ell-1}{2}-i,k}\left(\dfrac{\ell+1}{2}+i\right)\right] \\\notag
&=&\left[\sum\limits_{i=0}^{k/2}\left(\dfrac{\ell+1}{2}+i\right)\left(\dbinom{\ell}{k-2i}-\dbinom{\ell}{k-2i-2}\right)\right].
\end{eqnarray} The result follows as the sum telescopes.
\end{proof}

Our next goal is to state a corollary similar to Corollary~\ref{corollary:limit} when $\ell$ and $k$ get large. For this we let the ratio $k/\ell$ equal a fixed constant $b/a$ and consider the limit as $\ell$ and $k$ approach infinity. In the context of the merging problem, this is saying the percentage of red cars in a certain area is constant. Before stating this corollary, we need a helpful lemma first.

\begin{lemma}\label{lemma:helpful}
If $a\geq 2b$, then

\[
\lim\limits_{r\rightarrow\infty}\left.\sum\limits_{i=0}^{br}\dbinom{ar}{i}\middle/r\dbinom{ar}{br}\right.=0.
\]
\end{lemma}

\begin{proof}
First, suppose $a>2b$, then

\begin{eqnarray}
\notag\left.\sum\limits_{i=0}^{br}\dbinom{ar}{i}\middle/\dbinom{ar}{br}\right.&=&1+\dfrac{br}{ar-br+1}+\dfrac{br(br-1)}{(ar-br+1)(ar-br+2)}+\cdots\\\notag
&\leq&1+\dfrac{br}{ar-br+1}+\left(\dfrac{br}{ar-br+1}\right)^2+\cdots\\\notag
&=&\dfrac{ar-br+1}{ar-2br+1}. \\\notag
\end{eqnarray}

Thus, 

\[
\lim\limits_{r\rightarrow\infty}\left.\sum\limits_{i=0}^{br}\dbinom{ar}{i}\middle/r\dbinom{ar}{br}\right.\leq\lim\limits_{r\rightarrow\infty}\dfrac{ar-br+1}{r(ar-2br+1)}=0.
\]

Now suppose that $a=2b$. The above limit is not 0 in this case, so we handle it separately as follows.
\begin{eqnarray}
\notag\left.\sum\limits_{i=0}^{br}\dbinom{2br}{i}\middle/\dbinom{2br}{br}\right.&=&\left.\left(2^{2br-1}+\dfrac{1}{2}\dbinom{2br}{br}\right)\middle/\dbinom{2br}{br}\right.\\\notag
&=&\left.2^{2br-1}\middle/\dbinom{2br}{br}\right.+\dfrac{1}{2}\\\notag
&\sim&\dfrac{\sqrt{br\pi}+1}{2} \\\notag
\end{eqnarray} using Stirling's approximation. Thus,

\begin{equation*}
\lim\limits_{r\rightarrow\infty}\left.\sum\limits_{i=0}^{br}\dbinom{2br}{i}\middle/r\dbinom{2br}{br}\right.=\lim\limits_{r\rightarrow\infty}\dfrac{\sqrt{br\pi}+1}{2r}=0.
\qedhere \end{equation*}
\end{proof}

\begin{corollary}\label{corollary:limitk}
Let $\ell=ar$ and $k=br$ for positive integers $a$, $b$, and $r$. Then

\[
\lim\limits_{\ell\rightarrow\infty}\dfrac{\mathbb{E}[\ell,k]}{\ell}=\dfrac{1}{2}
\] when $\ell\geq2k$, and

\[
\lim\limits_{\ell\rightarrow\infty}\dfrac{\mathbb{E}[\ell,k]}{\ell}=\dfrac{b}{a}
\] when $\ell<2k$.
\end{corollary}

The proof follows since the sums in Theorem \ref{theorem:expectedk} are partial sums of the sum in Lemma \ref{lemma:helpful}. 

\begin{example}
Consider the case where there are the same even number of red and green cars; let $\ell=4n$ and $k=2n$. The second formula above simplifies to

\[
\frac{2n\dbinom{4n}{2n}+\sum\limits_{i=0}^{n}\dbinom{4n}{2i+1}}{\dbinom{4n}{2n}}=2n+\frac{2^{4n-2}}{\dbinom{4n}{2n}}=k+\frac{2^{2k-2}}{\dbinom{2k}{k}}.
\] The same simplification occurs when $k$ is odd.
\end{example}

\section{A Connection to Domino Snakes}\label{section:domino}
Recall that we let $B_{\ell,k}$ denote the set of arrival sequences with exactly $k$ zeros, and we let $R_{\ell,k}$ denote the sum of the number of cars in the right lane for all arrival sequences in $B_{\ell,k}$. We noted in Section~\ref{section:kredcars}
that the second column of Table~\ref{table:Rlk}, is the sequence $R_{\ell,1}$ that begins: \[1, 3, 6, 9, 15, 19, 28, 33, 45, 51, 66, 73, 91, 99.\] These numbers are listed in the Online Encyclopedia of Integer Sequences (OEIS) as sequence \href{https://oeis.org/A031940}{A031940}, and that entry states (without proof or citations) that this sequence describes the length of the longest legal domino snake using a full set of dominoes up to $[\ell:\ell]$, which we denote $D_\ell$, and the number $T_\ell$ of edges in a longest trail on the complete graph on $\ell$ vertices with loops, which we denote $K_\ell^\circ$ \cite{oeis}. A domino snake is a single line of dominoes laid out so that the ends match.  Example~\ref{example:snake} shows examples of some domino snakes and their corresponding trails in $K_4^\circ$.

       \begin{example} \label{example:snake}
    Let $n=4$. There will always be one domino leftover. Here is one possible longest snake of length $9$, and its corresponding path in $K_4^\circ$.
    \[ [4:1][1:1][1:3][3:2][2:2][2:4][4:4][4:3][3:3] \]
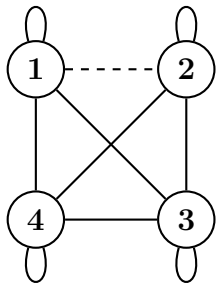
\begin{center}
\begin{tikzpicture}[auto,node distance=2cm,
                thick,main node/.style={circle,draw,font=\bfseries}, scale=.7]

  \node[main node] (1) {1};
  \node[main node] (2) [right of=1] {2};
  \node[main node] (4) [below of=1] {4};
  \node[main node] (3) [below of=2] {3};

  \draw  (1) edge[loop above] node {} (1);
   \draw  (2) edge[loop above] node {} (2);
  \draw  (3) edge[loop below] node {} (3);
  \draw  (4) edge[loop below] node {} (4);
  \draw   (1) [dashed] edge node {} (2);
 \draw    (1) edge node {} (3);  
  \draw   (1) edge node {} (4);  
  \draw   (2) edge node {}  (3);  
 \draw    (2) edge node {} (4);  
 \draw    (3) edge node {} (4); 
\end{tikzpicture}
  \end{center}
     \[4 \to 1 \to 1 \to 3 \to 3 \to 2 \to 2 \to 4 \to  4 \to 3 \to 3 \]

    \end{example}
    
In this section, we prove the following result, and we give an explicit bijection between the set of cars in the right lane of all arrival sequences of length $\ell$ with exactly one red car and the edges in a longest trail in the complete graph with loops $\Klo$.

\begin{theorem} \label{thm:RDT}
All of these sequences can be computed as: \[R_{\ell,1} = D_\ell = T_\ell =  \begin{cases} 
\binom{\ell}{2}+\ell & \text{ if } \ell \text{ is odd} \\
\binom{\ell}{2} + \frac{\ell}{2}+1 & \text{ if } \ell \text{ is even.}
\end{cases}\]
\end{theorem}

\begin{proof} 
When $\ell$ is odd, the degree of every vertex of $K_\ell$ is even, so there exists an Eulerian circuit of $K_\ell$. To include the loops, simply follow the loop each time a vertex is encountered for the first time in the trail. In $K_\ell$ with loops, there are $\binom{\ell}{2}+\ell$ edges and all are used in the longest trail.

When $\ell$ is even, construct a subgraph $H_\ell\le K_\ell$ by removing the edges $(1,2)$, $(3,4)$, \dots $(\ell-3,\ell-2)$. Then every vertex has even degree except the vertices $\ell-1$ and $\ell$. There exists an Eulerian trail of $H_\ell$ with $\frac{\ell(\ell-2)}{2}+1$ edges. The complete graph $K_\ell$ could not have a longer trail because every interior vertex of the trail must have even degree. When we add in the loops like above, we have a total of $\binom{\ell}{2}+\frac{\ell}{2}+1$ edges. We conclude that $T_\ell$ is described by the polynomials stated above.

Next consider the $\ell$ arrival sequences in $\Bl$. When $\ell= 2k+1$, each has $\frac{\ell+1}{2} $ cars in the right lane. Hence, $$ R_{\ell,1} = (\ell)\left(\frac{\ell+1}{2} \right ) = \frac{\ell^2+\ell}{2} = \binom{\ell}{2}+\ell.$$

   When $\ell=2k$, $\ell-1$ arrival sequences result in  $\frac{\ell}{2}$ cars in the right lane, and one results in $\frac{\ell}{2}+1$ cars.  Hence we calculate the right lane length 
 $R_{\ell,1} = (\ell-1)\left(\frac{\ell}{2}\right ) + \frac{\ell}{2}+1 =\binom{\ell}{2}+\frac{\ell}{2}+1.$ 
We conclude that $R_{\ell,1}$ is described by the polynomials stated above.

Finally, we note there is a natural bijection from the set of trails on $K_\ell^\circ$ to the set of domino snakes by considering an orientation of the trail and mapping each edge $(i,j)$ to a domino $[i:j]$, and this bijection shows that $D_\ell = T_\ell$ for $\ell \geq 1$.
\end{proof}

We remark that it is a fun exercise to also generate these domino snakes (and longest trails) recursively, and then to show the recursion satisfies the closed formula in Theorem~\ref{thm:RDT}, but we will not include that here.

Let $\Bl$ denote the set of arrival sequences with exactly 1 zero (or red car). The rest of this section is dedicated to describing a bijective map \[ \rho: \Bl \rightarrow \Klo \] that sends each car in the right lane of an arrival sequence in $\Bl$ to an edge in a longest trail in $\Klo.$

  The definition  of the map $\rho: \Bl \rightarrow \Klo$ depends on the parity of $\ell$ and the parity of the index $p$ where the unique zero in the arrival sequence $\bb$ in $\Bl$ appears. For instance, the tables in Figure~\ref{fig:tables} show the image of $\rho$ for each string in $\Bl$ for $\ell=6$ and $\ell = 7$.
The cars in the right lane of each arrival sequence are highlighted in color (red or blue), with the unique car corresponding to a zero in the arrival sequence highlighted in red.  (The blue entries are in fact bounces, which matches our previous notation.) We note that the edges $(1,2), (3,4), \ldots, (2k-3,2k-2)$ are not included in the image of $\rho$ when $\ell =2k$, but these edges are included in the image of $\rho$ when $\ell = 2k+1$.  

 \begin{figure}[h]

\begin{tabular}{|c|c|c|} \hline
$\bb$ &  $\rho(\bb)$ &  $r(\bb)$\\  \hline
\tcr{0}1\tcb{1}1\tcb{1}1     & $ \{\tcr{(1,1)},\tcb{(1,3),(1,5)} \}$ & 3 \\
\tcb{1}\tcr{0}11\tcb{1}1     & $ \{\cancel{(1,2)},\tcr{(2,2)},\tcb{(2,4),(2,6)} \}$ & 3 \\
\tcb{1}1\tcr{0}1\tcb{1}1     & $ \{\tcb{(2,3)},\tcr{(3,3)},\tcb{(3,5)} \}$ & 3 \\
\tcb{1}1\tcb{1}\tcr{0}11     & $ \{\tcb{(1,4)},\cancel{(3,4)},\tcr{(4,4)},\tcb{(4,6)} \}$ & 3 \\
\tcb{1}1\tcb{1}1\tcr{0}1     & $ \{\tcb{(2,5),(4,5)},\tcr{(5,5)} \}$ & 3 \\
\tcb{1}1\tcb{1}1\tcb{1}\tcr{0}     & $ \{\tcb{(1,6),(3,6),(5,6)},\tcr{(6,6)} \}$ & 4 \\
     &  &  \\
 \hline
\end{tabular}
\begin{tabular}{|c|c|c|} \hline
$\bb$ &  $\rho(\bb)$ &  $r(\bb)$\\  \hline
\tcr{0}1\tcb{1}1\tcb{1}1\tcb{1}     & $ \{\tcr{(1,1)},\tcb{(1,3),(1,5),(1,7)}  \}$ & 4 \\ 
\tcb{1}\tcr{0}11\tcb{1}1\tcb{1}     & $ \{\tcb{(1,2)},\tcr{(2,2)},\tcb{(2,4),(2,6)} \}$ & 4 \\
\tcb{1}1\tcr{0}1\tcb{1}1\tcb{1}     & $ \{\tcb{(2,3)},\tcr{(3,3)},\tcb{(3,5),(3,7)}  \}$ & 4 \\
\tcb{1}1\tcb{1}\tcr{0}11\tcb{1}     & $ \{\tcb{(1,4),(3,4)},\tcr{(4,4)},\tcb{(4,6)} \}$ & 4 \\
\tcb{1}1\tcb{1}1\tcr{0}1\tcb{1}     & $ \{\tcb{(2,5),(4,5)},\tcr{(5,5)},\tcb{(5,7)} \}$ & 4 \\
\tcb{1}1\tcb{1}1\tcb{1}\tcr{0}1     & $ \{\tcb{(1,6),(3,6),(5,6)},\tcr{(6,6)} \}$ & 4 \\
\tcb{1}1\tcb{1}1\tcb{1}1\tcr{0}    & $\{\tcb{(2,7),(4,7),(6,7)},\tcr{(7,7)} \}$  & 4\\
\hline
\end{tabular} 
\caption{Arrival sequences in $\Bl$ and their images under $\rho$ when $\ell=6,7$.} \label{fig:tables}
 \end{figure}

\noindent 
 If car $c$ is in the right lane of the arrival sequence with a zero in position $p$, then $\rho$ is defined as follows with the even $\ell$ on the left and the odd $\ell$ on the right. 
$$
\rho (c,p) = 
\begin{cases} (c+1,p) & c<p, p \text{ is odd } \\
(c,p) & c \leq p, c \neq p-1, p \text{ is even } \\
 (p,\ell) & c =p-1, p \text{ is even }\\
(p,c) & c \geq p, p \text{ is odd }\\
(p,c-1) & c >p, p \text{ is even} 
\end{cases} \hspace{12pt}
\rho (c,p) = 
\begin{cases} (c+1,p) & c<p, p \text{ is odd } \\
(p,p) & c =p \\
(p,c) & c >p, p \text{ is odd } \\
(c,p) & c<p, p \text{ is even }\\
(p,c-1) & c >p, p \text{ is even}
\end{cases}
$$

 \noindent 
 For the inverse map, $i\leq j$ for each edge $(i,j)$ in $\Klo$, again with even $\ell$ on the left and the odd $\ell$ on the right. 
$$
\rho^{-1} (i,j) = 
\begin{cases} (j,i) & i, j \text{ odd, } i\neq j\\
(i-1,i) & i, j \text{  even, } j=\ell \\ 
(j+1,i) & i, j \text{  even, } j\neq i,\ell \\
(i-1,j) & i \text{ is even, } j \text{ is odd} \\
(i,j) & i \text{ is odd, } j \text{ is even} \\
(i,j) & i=j
\end{cases}
\hspace{1cm}
\rho^{-1} (i,j) = 
\begin{cases} (j,i) & i, j \text{ odd, } i\neq j\\
(j+1,i) & i, j \text{  even, } i\neq j  \\
(i-1,j) & i \text{ is even, } j \text{ is odd} \\
(i,j) & i \text{ is odd, } j \text{ is even} \\
(i,j) & i=j
\end{cases}
$$

\begin{proposition} \label{prop:sets}
The map $\rho$ defines a bijection between the set of cars in the right lane of all the arrival sequences in $\Bl$ to the set of edges in a longest trail in $\Klo$.
Moreover, this bijection implies that  $R_{\ell,1} = T_\ell$ for $\ell \in \mathbb{N}$.
\end{proposition}
\begin{proof}
The map $\rho$ is invertible when restricted to its image, and Theorem~\ref{thm:RDT} proves that its image is the correct size of a longest trail in $\Klo.$  When $\ell$ is odd, the image
of $\rho$ contains all $\binom{\ell}{2}+\ell$ edges in $\Klo$, so it forms an Eulerian circuit in $\Klo.$ When $\ell$ is even, the image defines a longest trail in $\Klo$ because there are exactly two vertices in the image of $\rho$ that have odd degree, which are the vertices labeled $\ell$ and $\ell-1$, and moreover, vertex $\ell$ is adjacent to every other vertex. We conclude that the image of $\rho$ is a connected subgraph with exactly two vertices having odd degree so it forms a longest trail in $\Klo.$
\end{proof}

We complete this section by revisiting  Example~\ref{example:snake}.
\begin{example}
One can see that edges in $\rho(\bb)$ correspond to the edges in the longest trail 
$4 \to 1 \to 1 \to 3 \to 3 \to 2 \to 2 \to 4 \to  4 \to 3 \to 3 $
and also the longest domino snake   $ [4:1][1:1][1:3][3:2][2:2][2:4][4:4][4:3][3:3] $.
\begin{center}
    \begin{tabular}{|c|c|c|} \hline
$\bb$ &  $\rho(\bb)$ &  $r(\bb)$\\  \hline
\tcr{0}1\tcb{1}1     & $ \{\tcr{(1,1)},\tcb{(1,3)} \}$ & 2 \\ 
\tcb{1}\tcr{0}11     & $ \{\cancel{(1,2)} ,\tcr{(2,2)}, \tcb{(2,4)} \}$ & 2 \\
\tcb{1}1\tcr{0}1     & $ \{\tcb{(2,3)},\tcr{(3,3)} \}$ & 2 \\
\tcb{1}1\tcb{1}\tcr{0}    & $ \{\tcb{(1,4),(3,4)},\tcr{(4,4)} \}$ & 3 \\  \hline
\end{tabular}
\end{center}
\end{example}

\section{Color-blind equivalence classes}\label{section:equiv_classes}
Consider the two arrival sequences $ \aa=01110$ and $\bb=11110$. They are equivalent in the sense that the first, third, and fifth car in each arrival sequence end up in the right lane. So if the sides of the cars were labeled by their starting position in each arrival sequence, as depicted in Figure~\ref{fig:ar3}, a color-blind observer would not be able to differentiate their final structures.   
When we disregard color, the final structure of the cars is completely determined by the \textbf{right lane vector} $\rlv$ that records the order of the cars in the right lane. For instance, the two arrival sequences listed above have right lane vectors $\rlv(01110)=\rlv(11110)=(1,3,5)$.

\begin{figure} [h]
\begin{tikzpicture} [scale=0.9]
  \begin{scope}[scale=0.5]
     \shade[top color=red, bottom color=white, shading angle={45}]
    [draw=black,fill=red!20,rounded corners=1.2ex,very thick] (1.5,.5) rectangle (6.5,1.8);
    \draw[very thick, rounded corners=0.5ex,fill=black!20!blue!20!white,thick]  (2.5,1.8) -- ++(1,0.7) -- ++(1.6,0) -- ++(0.6,-0.7) -- (2.5,1.8);
    \draw[thick]  (4.2,1.8) -- (4.2,2.5);
    \draw[draw=black,fill=gray!50,thick] (2.75,.5) circle (.5);
    \draw[draw=black,fill=gray!50,thick] (5.5,.5) circle (.5);
    \draw[draw=black,fill=gray!80,semithick] (2.75,.5) circle (.4);
    \draw[draw=black,fill=gray!80,semithick] (5.5,.5) circle (.4);
    \draw[-,semithick] (0,3.5) -- (7.5,3.5);
    \draw[dashed,thick] (0,-0.5) -- (7.5,-0.5);
    \draw (4,1.2) node {\Large {1}};
  \end{scope}
   \begin{scope}[xshift=100,scale=0.5] 
    \shade[top color=green, bottom color=white, shading angle={45}]
    [draw=black,fill=red!20,rounded corners=1.2ex,very thick] (1.5,.5) rectangle (6.5,1.8);
    \draw[very thick, rounded corners=0.5ex,fill=black!20!blue!20!white,thick]  (2.5,1.8) -- ++(1,0.7) -- ++(1.6,0) -- ++(0.6,-0.7) -- (2.5,1.8);
    \draw[thick]  (4.2,1.8) -- (4.2,2.5);
    \draw[draw=black,fill=gray!50,thick] (2.75,.5) circle (.5);
    \draw[draw=black,fill=gray!50,thick] (5.5,.5) circle (.5);
    \draw[draw=black,fill=gray!80,semithick] (2.75,.5) circle (.4);
    \draw[draw=black,fill=gray!80,semithick] (5.5,.5) circle (.4);
    \draw[-,semithick] (0,3.5) -- (7,3.5);
    \draw[dashed,thick] (0,-0.5) -- (7,-0.5);
    \draw (4,1.2) node {\Large {3}};
  \end{scope}
  \begin{scope}[xshift=200,scale=0.5] 
    \shade[top color=red, bottom color=white, shading angle={45}]
    [draw=black,fill=red!20,rounded corners=1.2ex,very thick] (1.5,.5) rectangle (6.5,1.8);
    \draw[very thick, rounded corners=0.5ex,fill=black!20!blue!20!white,thick]  (2.5,1.8) -- ++(1,0.7) -- ++(1.6,0) -- ++(0.6,-0.7) -- (2.5,1.8);
    \draw[thick]  (4.2,1.8) -- (4.2,2.5);
    \draw[draw=black,fill=gray!50,thick] (2.75,.5) circle (.5);
    \draw[draw=black,fill=gray!50,thick] (5.5,.5) circle (.5);
    \draw[draw=black,fill=gray!80,semithick] (2.75,.5) circle (.4);
    \draw[draw=black,fill=gray!80,semithick] (5.5,.5) circle (.4);
    \draw[-,semithick] (0,3.5) -- (7,3.5);
    \draw[dashed,thick] (0,-0.5) -- (7,-0.5);
    \draw (4,1.2) node {\Large {5}};
  \end{scope}
 \end{tikzpicture}
 \vspace{0.5cm}
 \begin{tikzpicture} [scale=0.9]
  \begin{scope}[scale=0.5] 
     \shade[top color=green, bottom color=white, shading angle={45}]
    [draw=black,fill=red!20,rounded corners=1.2ex,very thick] (1.5,.5) rectangle (6.5,1.8);
    \draw[very thick, rounded corners=0.5ex,fill=black!20!blue!20!white,thick]  (2.5,1.8) -- ++(1,0.7) -- ++(1.6,0) -- ++(0.6,-0.7) -- (2.5,1.8);
    \draw[thick]  (4.2,1.8) -- (4.2,2.5);
    \draw[draw=black,fill=gray!50,thick] (2.75,.5) circle (.5);
    \draw[draw=black,fill=gray!50,thick] (5.5,.5) circle (.5);
    \draw[draw=black,fill=gray!80,semithick] (2.75,.5) circle (.4);
    \draw[draw=black,fill=gray!80,semithick] (5.5,.5) circle (.4);
    \draw[-,semithick] (0,-0.5) -- (7,-0.5);
    \draw (4,1.2) node {\Large {2}};
  \end{scope}
   \begin{scope}[xshift=100,scale=0.5] 
    \shade[top color=green, bottom color=white, shading angle={45}]
    [draw=black,fill=red!20,rounded corners=1.2ex,very thick] (1.5,.5) rectangle (6.5,1.8);
    \draw[very thick, rounded corners=0.5ex,fill=black!20!blue!20!white,thick]  (2.5,1.8) -- ++(1,0.7) -- ++(1.6,0) -- ++(0.6,-0.7) -- (2.5,1.8);
    \draw[thick]  (4.2,1.8) -- (4.2,2.5);
    \draw[draw=black,fill=gray!50,thick] (2.75,.5) circle (.5);
    \draw[draw=black,fill=gray!50,thick] (5.5,.5) circle (.5);
    \draw[draw=black,fill=gray!80,semithick] (2.75,.5) circle (.4);
    \draw[draw=black,fill=gray!80,semithick] (5.5,.5) circle (.4);
     \draw[-,semithick] (0,-0.5) -- (7,-0.5);
    \draw (4,1.2) node {\Large {4}};
  \end{scope}
  \begin{scope}[xshift=200,scale=0.5]
   \draw[-,semithick] (0,-0.5) -- (7,-0.5);
  \end{scope}
\end{tikzpicture}
\begin{tikzpicture} [scale=0.9]
  \begin{scope}[scale=0.5]
     \shade[top color=green, bottom color=white, shading angle={45}]
    [draw=black,fill=red!20,rounded corners=1.2ex,very thick] (1.5,.5) rectangle (6.5,1.8);
    \draw[very thick, rounded corners=0.5ex,fill=black!20!blue!20!white,thick]  (2.5,1.8) -- ++(1,0.7) -- ++(1.6,0) -- ++(0.6,-0.7) -- (2.5,1.8);
    \draw[thick]  (4.2,1.8) -- (4.2,2.5);
    \draw[draw=black,fill=gray!50,thick] (2.75,.5) circle (.5);
    \draw[draw=black,fill=gray!50,thick] (5.5,.5) circle (.5);
    \draw[draw=black,fill=gray!80,semithick] (2.75,.5) circle (.4);
    \draw[draw=black,fill=gray!80,semithick] (5.5,.5) circle (.4);
    \draw[-,semithick] (0,3.5) -- (7,3.5);
    \draw[dashed,thick] (0,-0.5) -- (7,-0.5);
    \draw (4,1.2) node {\Large {1}};
  \end{scope}
   \begin{scope}[xshift=100,scale=0.5] 
    \shade[top color=green, bottom color=white, shading angle={45}]
    [draw=black,fill=red!20,rounded corners=1.2ex,very thick] (1.5,.5) rectangle (6.5,1.8);
    \draw[very thick, rounded corners=0.5ex,fill=black!20!blue!20!white,thick]  (2.5,1.8) -- ++(1,0.7) -- ++(1.6,0) -- ++(0.6,-0.7) -- (2.5,1.8);
    \draw[thick]  (4.2,1.8) -- (4.2,2.5);
    \draw[draw=black,fill=gray!50,thick] (2.75,.5) circle (.5);
    \draw[draw=black,fill=gray!50,thick] (5.5,.5) circle (.5);
    \draw[draw=black,fill=gray!80,semithick] (2.75,.5) circle (.4);
    \draw[draw=black,fill=gray!80,semithick] (5.5,.5) circle (.4);
    \draw[-,semithick] (0,3.5) -- (7,3.5);
    \draw[dashed,thick] (0,-0.5) -- (7,-0.5);
    \draw (4,1.2) node {\Large {3}};
  \end{scope}
  \begin{scope}[xshift=200,scale=0.5] 
    \shade[top color=red, bottom color=white, shading angle={45}]
    [draw=black,fill=red!20,rounded corners=1.2ex,very thick] (1.5,.5) rectangle (6.5,1.8);
    \draw[very thick, rounded corners=0.5ex,fill=black!20!blue!20!white,thick]  (2.5,1.8) -- ++(1,0.7) -- ++(1.6,0) -- ++(0.6,-0.7) -- (2.5,1.8);
    \draw[thick]  (4.2,1.8) -- (4.2,2.5);
    \draw[draw=black,fill=gray!50,thick] (2.75,.5) circle (.5);
    \draw[draw=black,fill=gray!50,thick] (5.5,.5) circle (.5);
    \draw[draw=black,fill=gray!80,semithick] (2.75,.5) circle (.4);
    \draw[draw=black,fill=gray!80,semithick] (5.5,.5) circle (.4);
    \draw[-,semithick] (0,3.5) -- (7,3.5);
    \draw[dashed,thick] (0,-0.5) -- (7,-0.5);
    \draw (4,1.2) node {\Large {5}};
  \end{scope}
 \end{tikzpicture}
 \vspace{0.4cm}
 \begin{tikzpicture} [scale=0.9]
  \begin{scope}[scale=0.5]
     \shade[top color=green, bottom color=white, shading angle={45}]
    [draw=black,fill=red!20,rounded corners=1.2ex,very thick] (1.5,.5) rectangle (6.5,1.8);
    \draw[very thick, rounded corners=0.5ex,fill=black!20!blue!20!white,thick]  (2.5,1.8) -- ++(1,0.7) -- ++(1.6,0) -- ++(0.6,-0.7) -- (2.5,1.8);
    \draw[thick]  (4.2,1.8) -- (4.2,2.5);
    \draw[draw=black,fill=gray!50,thick] (2.75,.5) circle (.5);
    \draw[draw=black,fill=gray!50,thick] (5.5,.5) circle (.5);
    \draw[draw=black,fill=gray!80,semithick] (2.75,.5) circle (.4);
    \draw[draw=black,fill=gray!80,semithick] (5.5,.5) circle (.4);
    \draw[-,semithick] (0,-0.5) -- (7,-0.5);
    \draw (4,1.2) node {\Large {2}};
  \end{scope}
   \begin{scope}[xshift=100,scale=0.5] 
    \shade[top color=green, bottom color=white, shading angle={45}]
    [draw=black,fill=red!20,rounded corners=1.2ex,very thick] (1.5,.5) rectangle (6.5,1.8);
    \draw[very thick, rounded corners=0.5ex,fill=black!20!blue!20!white,thick]  (2.5,1.8) -- ++(1,0.7) -- ++(1.6,0) -- ++(0.6,-0.7) -- (2.5,1.8);
    \draw[thick]  (4.2,1.8) -- (4.2,2.5);
    \draw[draw=black,fill=gray!50,thick] (2.75,.5) circle (.5);
    \draw[draw=black,fill=gray!50,thick] (5.5,.5) circle (.5);
    \draw[draw=black,fill=gray!80,semithick] (2.75,.5) circle (.4);
    \draw[draw=black,fill=gray!80,semithick] (5.5,.5) circle (.4);
     \draw[-,semithick] (0,-0.5) -- (7,-0.5);
    \draw (4,1.2) node {\Large {4}};
  \end{scope}
  \begin{scope}[xshift=200,scale=0.5]
   \draw[-,semithick] (0,-0.5) -- (7,-0.5);
  \end{scope}
\end{tikzpicture}
\begin{tikzpicture} [scale=0.9]
  \begin{scope}[scale=0.5]
     \shade[top color=brown, bottom color=white, shading angle={45}]
    [draw=black,fill=red!20,rounded corners=1.2ex,very thick] (1.5,.5) rectangle (6.5,1.8);
    \draw[very thick, rounded corners=0.5ex,fill=black!20!blue!20!white,thick]  (2.5,1.8) -- ++(1,0.7) -- ++(1.6,0) -- ++(0.6,-0.7) -- (2.5,1.8);
    \draw[thick]  (4.2,1.8) -- (4.2,2.5);
    \draw[draw=black,fill=gray!50,thick] (2.75,.5) circle (.5);
    \draw[draw=black,fill=gray!50,thick] (5.5,.5) circle (.5);
    \draw[draw=black,fill=gray!80,semithick] (2.75,.5) circle (.4);
    \draw[draw=black,fill=gray!80,semithick] (5.5,.5) circle (.4);
    \draw[-,semithick] (0,3.5) -- (7,3.5);
    \draw[dashed,thick] (0,-0.5) -- (7,-0.5);
    \draw (4,1.2) node {\Large {1}};
  \end{scope}
   \begin{scope}[xshift=100,scale=0.5] 
    \shade[top color=brown, bottom color=white, shading angle={45}]
    [draw=black,fill=red!20,rounded corners=1.2ex,very thick] (1.5,.5) rectangle (6.5,1.8);
    \draw[very thick, rounded corners=0.5ex,fill=black!20!blue!20!white,thick]  (2.5,1.8) -- ++(1,0.7) -- ++(1.6,0) -- ++(0.6,-0.7) -- (2.5,1.8);
    \draw[thick]  (4.2,1.8) -- (4.2,2.5);
    \draw[draw=black,fill=gray!50,thick] (2.75,.5) circle (.5);
    \draw[draw=black,fill=gray!50,thick] (5.5,.5) circle (.5);
    \draw[draw=black,fill=gray!80,semithick] (2.75,.5) circle (.4);
    \draw[draw=black,fill=gray!80,semithick] (5.5,.5) circle (.4);
    \draw[-,semithick] (0,3.5) -- (7,3.5);
    \draw[dashed,thick] (0,-0.5) -- (7,-0.5);
    \draw (4,1.2) node {\Large {3}};
  \end{scope}
  \begin{scope}[xshift=200,scale=0.5] 
    \shade[top color=brown, bottom color=white, shading angle={45}]
    [draw=black,fill=red!20,rounded corners=1.2ex,very thick] (1.5,.5) rectangle (6.5,1.8);
    \draw[very thick, rounded corners=0.5ex,fill=black!20!blue!20!white,thick]  (2.5,1.8) -- ++(1,0.7) -- ++(1.6,0) -- ++(0.6,-0.7) -- (2.5,1.8);
    \draw[thick]  (4.2,1.8) -- (4.2,2.5);
    \draw[draw=black,fill=gray!50,thick] (2.75,.5) circle (.5);
    \draw[draw=black,fill=gray!50,thick] (5.5,.5) circle (.5);
    \draw[draw=black,fill=gray!80,semithick] (2.75,.5) circle (.4);
    \draw[draw=black,fill=gray!80,semithick] (5.5,.5) circle (.4);
    \draw[-,semithick] (0,3.5) -- (7,3.5);
    \draw[dashed,thick] (0,-0.5) -- (7,-0.5);
    \draw (4,1.2) node {\Large {5}};
  \end{scope}
 \end{tikzpicture}
 \vspace{0.4cm}
 \begin{tikzpicture} [scale=0.9]
  \begin{scope}[scale=0.5]
     \shade[top color=brown, bottom color=white, shading angle={45}]
    [draw=black,fill=red!20,rounded corners=1.2ex,very thick] (1.5,.5) rectangle (6.5,1.8);
    \draw[very thick, rounded corners=0.5ex,fill=black!20!blue!20!white,thick]  (2.5,1.8) -- ++(1,0.7) -- ++(1.6,0) -- ++(0.6,-0.7) -- (2.5,1.8);
    \draw[thick]  (4.2,1.8) -- (4.2,2.5);
    \draw[draw=black,fill=gray!50,thick] (2.75,.5) circle (.5);
    \draw[draw=black,fill=gray!50,thick] (5.5,.5) circle (.5);
    \draw[draw=black,fill=gray!80,semithick] (2.75,.5) circle (.4);
    \draw[draw=black,fill=gray!80,semithick] (5.5,.5) circle (.4);
    \draw[-,semithick] (0,-0.5) -- (7,-0.5);
    \draw (4,1.2) node {\Large {2}};
  \end{scope}
   \begin{scope}[xshift=100,scale=0.5] 
    \shade[top color=brown, bottom color=white, shading angle={45}]
    [draw=black,fill=red!20,rounded corners=1.2ex,very thick] (1.5,.5) rectangle (6.5,1.8);
    \draw[very thick, rounded corners=0.5ex,fill=black!20!blue!20!white,thick]  (2.5,1.8) -- ++(1,0.7) -- ++(1.6,0) -- ++(0.6,-0.7) -- (2.5,1.8);
    \draw[thick]  (4.2,1.8) -- (4.2,2.5);
    \draw[draw=black,fill=gray!50,thick] (2.75,.5) circle (.5);
    \draw[draw=black,fill=gray!50,thick] (5.5,.5) circle (.5);
    \draw[draw=black,fill=gray!80,semithick] (2.75,.5) circle (.4);
    \draw[draw=black,fill=gray!80,semithick] (5.5,.5) circle (.4);
     \draw[-,semithick] (0,-0.5) -- (7,-0.5);
    \draw (4,1.2) node {\Large {4}};
  \end{scope}
  \begin{scope}[xshift=200,scale=0.5]
   \draw[-,semithick] (0,-0.5) -- (7,-0.5);
  \end{scope}
  \draw [very thick] [<-](0,-1)--(2,-1) node[right]{\text{Direction of traffic}};
\end{tikzpicture}
\caption{Arrival sequences $01\tcb{1}10$ and $\tcb{1}1\tcb{1}10$ with same $\rlv=(1,3,5)$ and their final color-blind result.}  \label{fig:ar3}
\end{figure}
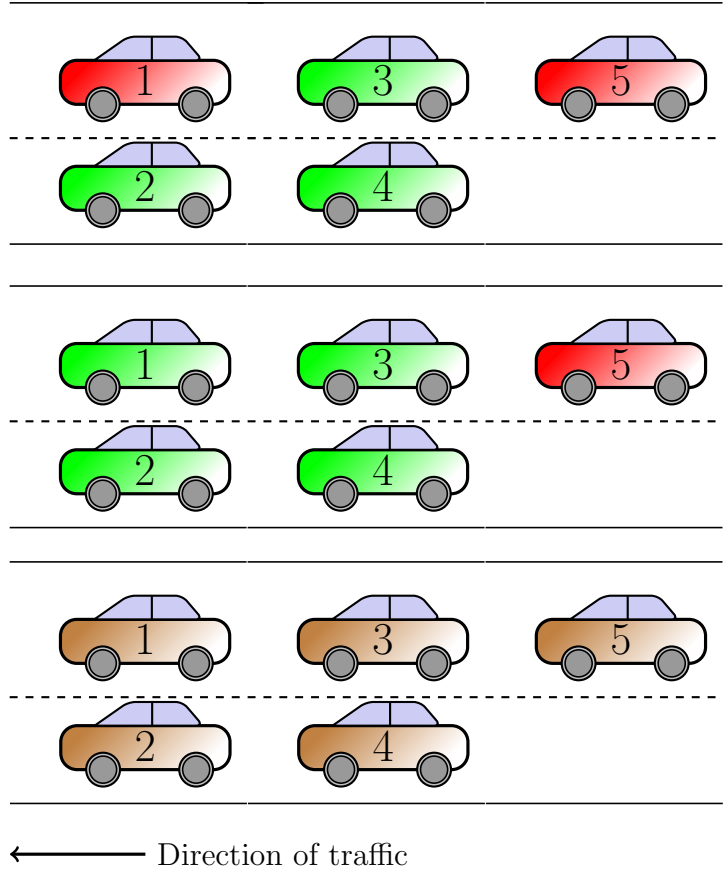

This observation leads us to define the following equivalence relation on the set of all arrival sequences $B_\ell$: two arrival sequences $\aa$ and $\bb$ are \textbf{color-blind equivalent} $\mathbf{a} \sim \mathbf{b}$ if $\rlv(\mathbf{a})=\rlv(\mathbf{b})$. Given an arrival sequence $\bb \in B_\ell$, let 
$$ \mathcal{C}(\bb) = \{ \aa \in B_\ell: \rlv(\aa) = \rlv(\bb) \} $$ denote the equivalence class of all arrival sequences who are color-blind equivalent to $\bb$.

The final structure of the right lane does not depend on whether the first car is red or green, since it will always stay in the right lane. 
Hence two binary strings that only differ in their first digit will have the same right lane vector $\rlv$, which implies the size of each color-blind equivalence class is even. 

\begin{proposition}\label{prop:even_class_size}
Each color-blind equivalence class $\mathcal{C}(\bb)$ has even number of elements. 
\end{proposition}

In fact, if two arrival sequences only differ in places where their merging paths are touching the diagonal, then they will be in the same color-blind equivalence class.  
Table~\ref{table:rlv} shows the color-blind equivalence classes for all arrival sequences in $B_6$, and it highlights where the merging path for each arrival sequence touches the diagonal in orange. 

\begin{table}[h]
\begin{tabular}{|c|c|c||c|c|c|} \hline 
Arrival sequence & Right lane & Class size  & Arrival sequence & Right lane & Class size \\
$\bb$ & $\rlv(\bb)$ & $|\mathcal{C}(\bb)|$ & $\bb$ & $\rlv(\bb)$ & $|\mathcal{C}(\bb)|$ \\ \hline 
$\tco{0}00000$ & \multirow{2}{*}{$1, 2, 3, 4, 5, 6$} & \multirow{ 2}{*}{$2$}  & $\tco{0}00001$ & \multirow{2}{*}{$1, 2, 3, 4, 5$} & \multirow{2}{*}{$2$} \\
$\tco{1}00000$ & &  &  $\tco{1}00001$ & &\\ \hline
$\tco{0}00010$ & \multirow{2}{*}{$1, 2, 3, 4, 6$} & \multirow{ 2}{*}{$2$} &
$\tco{0}00011$ & \multirow{2}{*}{$1, 2, 3, 4$} & \multirow{ 2}{*}{$2$} \\
$\tco{1}00010$ & & & $\tco{1}00011$ & &\\ \hline
$\tco{0}00100$ & \multirow{2}{*}{$1, 2, 3, 5, 6$} & \multirow{ 2}{*}{$2$} &
$\tco{0}00101$ & \multirow{2}{*}{$1, 2, 3, 5$} & \multirow{ 2}{*}{$2$} \\
$\tco{1}00100$ & & & $\tco{1}00101$ & &\\ \hline
$\tco{0}00110$ & \multirow{2}{*}{$1, 2, 3, 6$} & \multirow{ 2}{*}{$2$} &
$\tco{0}00111$ & \multirow{2}{*}{$1, 2, 3$} & \multirow{ 2}{*}{$2$} \\
$\tco{1}00110$ & & & $\tco{1}00111$ & &\\ \hline
$\tco{0}01000$ & \multirow{2}{*}{$1, 2, 4, 5, 6$} & \multirow{ 2}{*}{$2$} &
$\tco{0}01001$ & \multirow{2}{*}{$1, 2, 4, 5$} & \multirow{ 2}{*}{$2$} \\
$\tco{1}01000$ & & & $\tco{1}01001$ & &\\ \hline
$\tco{0}01010$ & \multirow{2}{*}{$1, 2, 4, 6$} & \multirow{ 2}{*}{$2$} & $\tco{0}01011$ & \multirow{2}{*}{$1, 2, 4$} & \multirow{ 2}{*}{$2$} \\
$\tco{1}01010$ & & & $\tco{1}01011$ & &\\ \hline
$\tco{0}011\tco{0}0$ & \multirow{4}{*}{$1, 2, 5, 6$} & \multirow{ 4}{*}{$4$} &$\tco{0}011\tco{0}1$ & \multirow{4}{*}{$1, 2, 5$} & \multirow{ 4}{*}{$4$} \\ 
$\tco{0}011\tco{1}0$ & & & $\tco{0}011\tco{1}1$ & &\\
$\tco{1}011\tco{0}0$ & & & $\tco{1}011\tco{0}1$ & &\\
$\tco{1}011\tco{1}0$ & & & $\tco{1}011\tco{1}1$ & &\\\hline
$\tco{0}1\tco{0}000$ & \multirow{4}{*}{$1, 3, 4, 5, 6$} & \multirow{ 4}{*}{$4$} & $\tco{0}1\tco{0}001$ & \multirow{4}{*}{$1, 3, 4, 5$} & \multirow{ 4}{*}{$4$} \\
$\tco{0}1\tco{1}000$ & & & $\tco{0}1\tco{1}001$ & &\\
$\tco{1}1\tco{0}000$ & & & $\tco{1}1\tco{0}001$ & &\\
$\tco{1}1\tco{1}000$ & & &  $\tco{1}1\tco{1}001$ & &\\\hline
$\tco{0}1\tco{0}010$ & \multirow{4}{*}{$1, 3, 4, 6$} & \multirow{ 4}{*}{$4$} & $\tco{0}1\tco{0}011$ & \multirow{4}{*}{$1, 3, 4$} & \multirow{ 4}{*}{$4$} \\
$\tco{0}1\tco{1}010$ & & & $\tco{0}1\tco{1}011$ & &\\
$\tco{1}1\tco{0}010$ & & & $\tco{1}1\tco{0}011$ & &\\
$\tco{1}1\tco{1}010$ & & & $\tco{1}1\tco{1}011$ & &\\\hline
$\tco{0}1\tco{0}1\tco{0}0$ & \multirow{8}{*}{$1, 3, 5, 6$} & \multirow{ 8}{*}{$8$} & $\tco{0}1\tco{0}1\tco{0}1$ & \multirow{8}{*}{$1, 3, 5$} & \multirow{ 8}{*}{$8$} \\
$\tco{0}1\tco{0}1\tco{1}0$ & & & $\tco{0}1\tco{0}1\tco{1}1$ & &\\
$\tco{0}1\tco{1}1\tco{0}0$ & & & $\tco{0}1\tco{1}1\tco{0}1$ & &\\
$\tco{0}1\tco{1}1\tco{1}0$ & & & $\tco{0}1\tco{1}1\tco{1}1$ & &\\
$\tco{1}1\tco{0}1\tco{0}0$ & & & $\tco{1}1\tco{0}1\tco{0}1$ & &\\
$\tco{1}1\tco{0}1\tco{1}0$ & & & $\tco{1}1\tco{0}1\tco{1}1$ & &\\
$\tco{1}1\tco{1}1\tco{0}0$ & & & 
$\tco{1}1\tco{1}1\tco{0}1$ & &\\
$\tco{1}1\tco{1}1\tco{1}0$ & & & $\tco{1}1\tco{1}1\tco{1}1$ & &\\\hline
\end{tabular} 
\caption{Arrival sequences in $B_6$ partitioned into color-blind equivalence classes. Touches are highlighted in orange.}
\label{table:rlv}
\end{table}

Let $\tv(\bb)$ denote the vector recording the steps where the merging path of the arrival sequence $\bb$ starts off touching the diagonal $x=y$, and $\touch(\bb)$ be the number of times $\bb$ touches the diagonal.  For instance $$ \tv(\tco{1}1\tco{1}1\tco{0}0) =\tv(\tco{0}1\tco{0}1\tco{0}0)=(1,3,5) \text{ with } \touch(\tco{1}1\tco{1}1\tco{0}0)=\touch(\tco{0}1\tco{0}1\tco{0}0)=3,$$  and $$ \tv(\tco{0}011\tco{1}0) = \tv(\tco{1}011\tco{1}0) = (1,5) \text{ with } \touch(\tco{0}011\tco{1}0) =\touch(\tco{1}011\tco{1}0) =2.$$
Our main result in this section shows that the size of a color-blind equivalence class $\mathcal{C}(\bb)$ depends only on $\touch(\bb)$.
\begin{theorem} \label{th:touch}
Let $\bb \in B_\ell$ be any arrival sequence of length $\ell$. Let $\touch=\touch(\bb)$ be the number of times $\bb$ touches the diagonal. Then the color-blind equivalence class $\mathcal{C}(\bb)$ contains $2^{\touch(\bb)}$ arrival sequences, or more succinctly, 
$$ |\mathcal{C}(\bb)|  = 2^\touch .$$ 
\end{theorem}

\begin{proof}
Every time a merging path is resting on the diagonal, the next car in the arrival sequence is forced into the right lane.  So if $\bb \in B_\ell$ has $\tv(\bb) = (1, j_2, j_3,\ldots, j_d)$ with $\touch(\bb) = d$, then we can switch the colors of any the $d$ cars in positions $\tv(\bb) = (1, j_2, j_3,\ldots, j_d)$ that were forced into the right lane after $\bb$ touched the diagonal to obtain a new merging sequence with the same right lane vector.  Conversely, when a merging path is above the diagonal, changing the color of the next car always results in a different right lane vector, because when the merging path is above the diagonal, green cars always go into the left lane and red cars always go to the right. 
We conclude that to construct $\mathcal{C}(\bb)$ combinatorially, we simply form all arrival sequences $\aa$ that agree with $\bb$ in all positions outside of $\tv(\bb)$, and we have $2^{\touch(\bb)}$ distinct choices for parity/color of the cars in positions $\tv(\bb)$. 
\end{proof}

\section{Future Work}

There are many possible variations on this problem, some developed by waiting in real-life traffic (as with the original problem) and some more abstract variations that may only apply to higher-dimensional traffic jams. We encourage anyone pursuing these problems to make good use of the OEIS, as we were frequently (pleasantly) surprised at the myriad connections to other areas of combinatorics.

Our first open problem considers the possibility that red and green cars are not evenly distributed in the arrival sequence, with green cars more likely to appear earlier in the sequence. This corresponds to the notion that drivers who pick the shortest lane are also the faster drivers.

\begin{oquestion}
Corollary \ref{corollary:limit} and Corollary \ref{corollary:limitk} give unsurprising results about what will happen to the expected length of the right lane as the number of cars gets large with the percentage of red cars held constant. How does this expected value change when we weight the arrival sequences so that sequences with more green cars in the front have a larger weight (probability of occurring)?
\end{oquestion}

Our work in this article has focused solely on two lanes merging. In reality, we may encounter three or more lanes, and there are choices for how to represent this mathematically. The next open problem describes one of these possibilities.

\begin{oquestion}
Consider three lanes merging into a single right lane with three types of drivers:
\begin{itemize}
    \item Those that pick the right lane only.
    \item Those that pick the shortest of the two right lanes.
    \item Those that pick the shortest of all the lanes.
\end{itemize}
How many of these merging paths reach the point $(m_1,m_2,m_3)$? What is the expected length of the right lane for all arrival sequences with a specific number of drivers of each type? This could also be extended to any number of lanes.
\end{oquestion}

Once the problem has been generalized to more lanes, it's natural to ask whether any of the connections to other combinatorial objects remains. The next two open questions address a couple of these connections.

\begin{oquestion}
In Section \ref{section:domino}, we mapped a subset of arrival sequences to the longest trail in a complete graph with loops. Can we generalize this to trails in hypergraphs when there are more than two lanes? There are multiple ways to define a trail in a hypergraph \cite{B17,L10}, but very little is currently known about the length of the longest trail or its connection to other combinatorial objects.
\end{oquestion}

\begin{oquestion}
In Section~\ref{section:bijection}, we found a connection to the expected maximum number of heads or tails in a set of coin flips. With $n$ lanes, is there a connection to the expected maximum number a face appears in $\ell$ rolls of an $n$-sided die?
\end{oquestion}

Another line of inquiry asks whether we are representing drivers appropriately as red or green drivers. More likely, each individual acts as a green driver with some fixed probability. However, representing drivers overall with this dichotomy likely mimics each person's individual likelihood of choosing the left or right lane. A different situation arises if we imagine drivers only choose the left lane if it is ``much'' shorter than the right lane. How much shorter? We could fix it at a certain number of cars, where our work so far has consisted of the case where green drivers choose the left when it is at least 1 car shorter. Or we could let it depend on the individual driver, leading to the following open question.

\begin{oquestion}
Consider the merging problem where green cars only choose the left lane if it is $c>1$ cars shorter than the right lane. Alternately, suppose car $i$ is associated with a value $c_i\in\mathbb{N}\cup\{\infty\}$ so that car $i$ will only choose the left lane if the difference between the lane lengths is at least $c_i$.
\end{oquestion}

Our final questions considers the case where the left lane has a fixed length $m$, so that once the right lane fills up with $m$ cars, no more cars can enter that lane. How is this real-world condition affecting cars who would like to move into that lane but are unable to?

\begin{oquestion}
Consider the merging problem where each lane has a capacity of $m$ cars and arrival sequences of length $2m-1$. Once the right lane reaches $m$ cars, then no more cars will be able to enter the left lane. What is the expected number of cars missing from the left lane?  
\end{oquestion}

\bibliography{Bibliography}
\bibliographystyle{plain}
\end{document}